\documentclass[12pt]{amsart}
\usepackage[utf8]{inputenc}

\usepackage{mathrsfs} 
\usepackage{amssymb}
\usepackage{bm}
\usepackage{graphicx}
\usepackage[centertags]{amsmath}
\usepackage{amsfonts}
\usepackage{amsthm}
\usepackage{enumitem}
\usepackage{tocvsec2}
\usepackage{xcolor}
\usepackage[margin=1in]{geometry}

\newtheorem{theorem}{Theorem}[section]
\newtheorem*{theorem*}{Theorem}

\newtheorem{corollary}[theorem]{Corollary}
\newtheorem{lemma}[theorem]{Lemma}
\newtheorem{rem}[theorem]{Remark}

\newtheorem{proposition}[theorem]{Proposition}

\newtheorem*{fact*}{Fact}

\theoremstyle{definition}

\newcommand{\ee}{\varepsilon}
\newcommand{\nn}{\mathbb{N}}

\newcommand{\tim}{\hat{\otimes}}

\begin{document}

\title{Higher projective tensor products of $c_0$}

\author{R.M. Causey}
\email{rmcausey1701@gmail.com}

\author{S.J. Dilworth} 
\address{Department of Mathematics, University of South Carolina, Columbia, SC
29208, U.S.A.}
\email{dilworth@math.sc.edu}

\begin{abstract}{\color{black} Let $m,n$ be positive integers with $m<n$. Under certain assumptions on the Banach space $X$, we prove that the $n$-fold projective tensor product of $X$, $\widehat{\otimes}^n_\pi X$, is not isomorphic to any subspace of any quotient of  the $m$-fold projective tensor product,  $\widehat{\otimes}_\pi^m X$. 
 In particular, we prove that   $\widehat{\otimes}^n_\pi c_0$ is not isomorphic to any subspace of any quotient of $\widehat{\otimes}_\pi^m c_0$.  This answers a question from \cite{CGS}. }

\end{abstract}

\maketitle

\section{Introduction}

Despite their fundamental role in duality in Banach space theory first explored by Grothendieck {\color{black}\cite{G}}, the theory of projective tensor products still has many fundamental unanswered questions.  Similarly, for $N>2$, {\color{black} there are} natural yet unanswered questions regarding  $N$-fold projective tensor products and their symmetric counterparts.  

  In this paper, we seek to extend recent results from \cite{CGS} and \cite{CGS2} which answer isomorphic questions regarding higher order tensor products using asymptotic properties related to asymptotically uniformly smooth renormings and the Szlenk index.   We recall that the Szlenk index is an ordinal index defined on the class of Asplund Banach spaces which quantifies how Asplund a Banach space is.  We also recall that the class of Banach spaces $X$ with $Sz(X)=\omega$ is the class of infinite dimensional space which admit equivalent, asymptotically uniformly smooth norms. Among these spaces, one can also define the Szlenk power type, which is a finer measure of the degree of asymptotic uniform smoothability of a given Banach space related to the optimum modulus of asymptotic uniform smoothness under equivalent norms.
  
     In \cite{C1}, an analogous stratification is discussed. For a given natural number $n$ and exponent $p\in (1,\infty]$, a property $\mathfrak{A}_{n,p}$ is described. The classes $\mathfrak{A}_{n,p}$, $1<p\leqslant \infty$, are subclasses of the class of Banach spaces $X$ such that $Sz(X)=\omega^{n+1}$, and the parameter $p$ allows for a finer stratification of this class. These are parallel to, but not exact analogues of, the stratification of the class of Banach spaces with Szlenk index $\omega$ using the notion of $p$-asymptotic uniform smoothability.  We also note that by H\'{o}lder duality, the class $\mathfrak
  {A}_{n,2}$ is halfway between the classes $\mathfrak{A}_{n,\infty}$ and $\mathfrak{A}_{n+1,\infty}$, as the class $\mathfrak{A}_{n+1,\infty}$ is halfway between the classes $\mathfrak{A}_{n,2}$ and $\mathfrak{A}_{n+1,2}$.  An intuitive way of understanding this is to consider the fact that  $\mathfrak{A}_{n,p}\subset \mathfrak{A}_{m,q}$ if and only if $n+1/p\leqslant m+1/q$, so the quantity $n+1/p$ encodes the strength of the property of being a member of $\mathfrak{A}_{n,p}$.  Since being a member of $\mathfrak{A}_{n,p}$ is a property which is stable under passing to subspaces, quotients, and isomorphs, the study of these classes is a powerful tool in the isomorphic theory of Banach spaces.  
  
  The fundamental idea behind the lower estimates in this work is that, due to Dvoretzky's theorem {\color{black} on almost euclidean subspaces \cite{Dv}}, each time we tensor with another infinite dimensional space, the best possible asymptotic property enjoyed by the projective tensor product goes up by at least one half step. More precisely, if $X_1, \ldots, X_n$ are infinite dimensional Banach spaces and $X_1\widehat{\otimes}_\pi \ldots \widehat{\otimes}_\pi X_n$ possesses property $\mathfrak{A}_{m,p}$, then $m+1/p \geqslant (n-1)/2$. 
  
  {\color{black} Under some mild assumptions regarding approximation properties of the spaces in question, we prove a converse of the latter result for Banach spaces whose duals have non-trivial cotype. More precisely,} we  show that if the spaces $X_1, \ldots, X_n$ each have property $\mathfrak{A}_{0,\infty}$ and their duals have cotype $2$, then the space $X_1\widehat{\otimes}_\pi \ldots \widehat{\otimes}_\pi X_n$ possesses property $\mathfrak{A}_{m,p}$ if and only if $m+1/p \geqslant (n-1)/2$. We obtain weaker but similar results by relaxing the assumption of cotype $2$ to non-trivial cotype. 
  
 {\color{black}  From these results  we deduce our main theorem in  Section~\ref{sec: mainresults}  below. In stating this result, we note that the property $\mathfrak{A}_{0,\infty}$ is equivalent to the property of being Asymptotic $c_0$ and that the relevant definition of a $\Pi_\lambda$ space, which is the natural approximation property relevant to the study of asymptotic properties, will be given in Section~\ref{sec: upperestimates}. }

\begin{theorem} For natural numbers $m,n$ with $m<n$, if $X_1, \ldots, X_m$ are Asymptotic $c_0$, $\Pi_\lambda$ Banach spaces such that $X_1^*, \ldots, X^*_m$ have non-trivial cotype, and if $Y_1,\ldots, Y_n$ are any infinite dimensional Banach spaces, then $Y_1\widehat{\otimes}_\pi \ldots \widehat{\otimes}_\pi Y_n$ is not isomorphic to any subspace of any quotient of $X_1\widehat{\otimes}_\pi \ldots \widehat{\otimes}_\pi X_m$.  In particular, $\widehat{\otimes}_\pi^n c_0$ is not isomorphic to any subspace of any quotient of $\widehat{\otimes}_\pi^m c_0$, and $\widehat{\otimes}_\pi^n T$ is not isomorphic to any subspace of any quotient of $\widehat{\otimes}_\pi^m T$, where $T$ denotes Tsirelson's space. 
\label{m1}
\end{theorem}

We also note the following immediate consequence of the preceding result, which follows by duality. 

\begin{corollary} For natural numbers $m,n$ with $m<n$, $\widehat{\otimes}_\ee^n \ell_1$ is not isomorphic to any subspace of $\widehat{\otimes}_\pi^m \ell_1$.  
\label{15}
\end{corollary}

It is worth noting that Corollary \ref{15} does not mention quotients. Indeed, $\widehat{\otimes}_\ee^n\ell_1$ is isometrically isomorphic to a quotient of $\ell_1$, and therefore of $\widehat{\times}_\ee^m \ell_1$ for any $m\in\nn$.   We discuss the incomplete dualization of Theorem \ref{m1}.   

Deducing Corollary \ref{15} from Theorem \ref{m1} uses the following duality result.  

\begin{theorem}{\color{black} Let $n$ be a natural number. If $ X_1, \ldots, X_n$ are Asymptotic $c_0$, $\Pi_\lambda$ Banach spaces such that $X_1^*, \ldots, X^*_n$ have non-trivial cotype, then $(X_1\widehat{\otimes}_\pi \ldots \widehat{\otimes}_\pi X_n)^* \approx X^*_1\widehat{\otimes}_\ee  \ldots \widehat{\otimes}_\ee X^*_n$.}    
\label{m2}
\end{theorem}

We would like to thank Richard Aron for bringing questions answered in this work to our attention, and for useful discussions on these topics.

\section{Tensor products}

All Banach spaces are over the field $\mathbb{K}$, which is either $\mathbb{R}$ or $\mathbb{C}$. Unless otherwise specified, all Banach spaces will be assumed to be infinite dimensional. Throughout, ``operator'' shall mean ``bounded, linear operator'' and ``subspace'' shall mean ``closed subspace.''  We let $B_X$ and $S_X$ denote the closed unit ball and closed unit sphere of $X$, respectively.   

For Banach spaces $X,Y$, we let $X\tim_\pi Y$ and $X\tim_\ee Y$ denote the completed projective and injective tensor products of $X$ and $Y$, respectively.   The norm  $\|\cdot\|_\pi$ is defined on the algebraic tensor product $X\otimes Y$  by \[\|u\|_\pi = \inf\Bigl\{\sum_{i=1}^n \|x_i\|\|y_i\|:n\in\nn, x_i\in X, y_i\in Y, u=\sum_{i=1}^n x_i\otimes y_i\Bigr\}.\]  Of course, the closed unit ball of $X\tim_\pi Y$ is the closed, convex hull of $\{x\otimes y: x\in B_X, y\in B_Y\}$.     The norm of $X\tim_\ee Y$ is obtained by treating each member $u=\sum_{i=1}^n x_i\otimes y_i$ of $X\otimes Y$ as a linear functional on $X^*\tim_\pi Y^*$ via the action \[\langle u, x^*\otimes y^*\rangle = \sum_{i=1}^n \langle x^*, x_i\rangle \langle y^*, y_i\rangle\] and defining \[\|u\|_\ee = \sup\Bigl\{|\langle u, x^*\otimes y^*\rangle|: x^*\otimes y^*\in B_{X^*\tim_\pi Y^*}\Bigr\}.\]  From this it is easy to see that $X^*\tim_\ee Y^*$ is isometrically contained in $(X\tim_\pi Y)^*$.  

We also define the $n$-fold projective and injective tensor product of a Banach space $X$. For convenience, we let $\tim_\pi^1 X=\tim_\ee^1 X=X$, and we define $\tim_\pi^{N+1} X= X\tim_\pi \bigl(\tim_\pi^N X\bigr)$, $\tim_\ee^{N+1} X= X\tim_\ee \bigl(\tim_\ee^N X\bigr)$.    Of course, $\tim_\ee^N X^*$ is isometrically included in $\bigl(\tim_\pi^N X\bigr)^*$.    We denote $x_1\otimes \ldots \otimes x_N$ by $\otimes_{n=1}^N x_n$. Given $x\in X$, we let $\otimes^N x=x\otimes \ldots \otimes x$, where the factor $x$ appears $N$ times. 

We define the symmetrization operator $S_N$ from the uncompleted $N$-fold tensor product $\tim^N X$ into itself which is the linear extension of \[S_N \otimes_{n=1}^N x_n = \frac{1}{N!}\sum_{\sigma\in \Sigma_N} \otimes_{n=1}^N x_{\sigma(n)},\] where $\Sigma_N$ denotes the set of all permutations on $\{1, \ldots, N\}$.   We let $\tim_s^N X$ denote the range in $\tim^N X$ of $S_N$.  Note that $S_N$ is a bounded projection on $\tim^N X$ with respect to both the projective and injective tensor norms, from which it follows that $S_N$ extends to an operator, which we also denote by $S_N$, from $\tim_\pi^N X$ to $\tim_\pi^N X$, as well as to an operator, also denoted by $S_N$, from $\tim_\ee^N X$ to $\tim_\ee^N X$.   We let $\tim_{\pi,s}^N X$ denote the range of $\tim_\pi^N X$ under the operator $S_N$, which we endow with the equivalent norm \[\|u\|_{\pi,s}=\inf\Bigl\{\sum_{i=1}^n |\lambda_i|\|x_i\|^N: u=\sum_{i=1}^n \lambda_i \tim^N x_i, x_i\in X\Bigr\}.\]   We let $\tim_{\ee,s}^N X$ denote the range of $\tim_\ee^N X$ under $S_N$ and endow it with the equivalent norm \[\|u\|=\{|\langle u, \tim^N x\rangle|: \tim^N x\in B_{\tim_{\pi,s}^N X}\Bigr\}.\]   Obviously $\tim_{\ee,s}^N X^*$ embeds isometrically into $\bigl(\tim_{\pi,s}^N X\bigr)^*$.   
{\color{black}We refer the reader to \cite{F} for further information regarding  the symmetric projective and injective  tensor products.}

We next recall a few properties of tensor products.

\begin{proposition}  \begin{enumerate}[label=(\roman*)]\item If $X^*$ has AP, $\mathfrak{K}(X,Y^*)=X^*\tim_\ee Y^*$. \item  $(X\tim_\pi Y)^*=\mathfrak{L}(X,Y^*)$.  \item If $X,Y$ have Schur property, so does $X\tim_\ee Y$. \item  If $X$ contains no isomorph of $\ell_1$ and if $Y^*$ has the Schur property, then $\mathfrak{L}(X,Y^*)=\mathfrak{K}(X,Y^*)$.\end{enumerate}

\label{sm}
\end{proposition}

\begin{proof}$(i)$ This is \cite[Corollary $4.13$]{RR}. 

$(ii)$ This is discussed in \cite[Page $24$]{RR}.

$(iii)$ This is due to Lust \cite{Lu}. 

For $(iv)$, if there existed $T:X\to Y^*$ which failed to be compact, then there would exist some $\ee>0$ and a bounded sequence $(x_n)_{n=1}^\infty\subset X$, which, since $X$ contains no isomorph of $\ell_1$, we can assume is weakly Cauchy, such that $\inf_{m\neq n}\|Tx_m-Tx_n\|=\ee>0$.  Then $(Tx_{2n}-Tx_{2n-1})_{n=1}^\infty$ is weakly null and not norm null in $Y^*$, contradicting the assumption that $Y^*$ has the Schur property.

\end{proof}

We note that $\bigl(\tim_\pi^N X\bigr)^*$ is the space $\mathcal{L}_N(X)$ of all bounded, $N$-linear functionals on $X^N$, and $\bigl(\tim_{\pi,s}^N X\bigr)^*$ is the space $\mathcal{P}_N(X)$ of all bounded, $N$-homogeneous polynomials on $X$. 

\begin{lemma} Let $X$ be a Banach space. If for some $N\in\nn$, $\mathcal{L}_N(X)=\tim_\ee^N X^*$, then $\mathcal{P}_N(X)=\tim_{\ee,s}^N X^*$.   
\label{l1}
\end{lemma}

\begin{proof} Assume $\mathcal{L}_N(X)=\tim_\ee^N X^*$. The symmetrization operator on $\tim_\pi^N X$ is a continuous projection onto $\tim_{\pi,s}^N X$, and therefore the adjoint is a surjection from $\mathcal{L}_N(X)$ to $\mathcal{P}_N(X)$.    But the range of the adjoint on $\mathcal{L}_N(X)=\tim_\ee^N X^*$ is just the symmetrization operator on $\tim_\ee^N X^*$, which has range $\tim_{\ee,s}^N X^*$.

\end{proof}

We note that if $X,Y$ are Banach spaces and $S:X\to X_1$, $T:Y\to Y_1$ are operators, then $S\otimes T:X\otimes Y\to X_1\otimes Y_1$ extends to an operator both from $X\tim_\pi Y$ to $X_1\tim_\pi Y_1$, and to an operator from $X\tim_\ee Y$ to $X_1\tim_\ee Y_1$. We denote both extensions by $S\otimes T$.  The norm of the extensions between either the projective or injective tensor products is $\|S\|\|T\|$. We also recall that injective tensor products respect subspaces, in the sense that if $S:X\to X_1$, $T:Y\to Y_1$ are isomorphic embeddings then so is $S\otimes T:X\tim_\ee Y\to X_1\tim_\ee Y_1$.

For $1\leqslant p<\infty$, a Banach space $X$, and a sequence $(x_n)_{n=1}^\infty\in \ell_\infty(X)$, we define the \emph{weakly} $p$-\emph{summing} norm of $(x_n)_{n=1}^\infty$ by  \[\|(x_n)_{n=1}^\infty\|_p^w = \sup\Bigl\{\Bigl(\sum_{n=1}^\infty |x^*(x_n)|^p\Bigr)^{1/p}: x^*\in B_{X^*}\Bigr\}.\]  We note that in this definition, it is sufficient to take the supremum over a $1$-norming subset $K$ of $B_{X^*}$.  This will be a convenient observation when evaluating the weakly $p$-summing norm of a sequence of tensors in the injective tensor product $X\tim_\ee Y$, since the set $K=\{x^*\otimes y^*: x^*\in B_{X^*}, y^*\in B_{Y^*}\}$ is $1$-norming.   We also use the notation $\|\cdot\|_p^w$ to denote the weakly $p$-summing norm of finite sequences: \[\|(x_n)_{n=1}^m\|_p^w = \sup\Bigl\{\Bigl(\sum_{n=1}^m |x^*(x_n)|^p\Bigr)^{1/p}: x^*\in B_{X^*}\Bigr\}.\]   We let \[\|(x_n)_{n=1}^\infty\|_\infty=\sup_n \|x_n\|,\] and we also use the notation $\|\cdot\|_\infty$ for finite sequences. We note that, with the usual change of notation when $p=\infty$, we can also define $\|(x_n)_{n=1}^\infty\|_\infty^w$ and $\|(x_n)_{n=1}^\infty\|_\infty^w$, but these coincide with $\|(x_n)_{n=1}^\infty\|_\infty$ and $\|(x_n)_{n=1}^\infty\|_\infty$, respectively. We will occasionally use the former notations for efficiency of presentation.

We note that if $1/p+1/q=1$, then $(x_n)_{n=1}^\infty$ is weakly $p$-summing if and only if the function $T:\ell_q\to X$ which takes $e_n$ to $x_n$ extends to a bounded, linear operator from $\ell_q$ to $X$, and the norm of the extension is equal to $\|(x_n)_{n=1}^\infty\|_p^w$.   Similarly, for a finite sequence $(x_n)_{n=1}^m$, $\|(x_n)_{n=1}^m\|_p^w$ is the norm of the operator $T:\ell_q^m\to X$ given by $T\sum_{n=1}^m a_ne_n=\sum_{n=1}^m a_nx_n$.    Given a sequence $(x_n\otimes y_n)_{n=1}^\infty$ of tensors, we use the notations $\|(x_n\otimes y_n)_{n=1}^\infty\|_{\pi,p}^w$ and $\|(x_n\otimes y_n)_{n=1}^\infty\|_{\ee,p}^w$ to indicate whether the tensors are considered as members of the projective or injective tensor product.

\begin{proposition} Let $X,Y,Z$ be Banach spaces. Fix sequences $(x_n)_{n=1}^\infty\subset X$, $(y_n)_{n=1}^\infty\subset Y$, and $(z_n)_{n=1}^\infty\subset Z$.  \begin{enumerate}[label=(\roman*)]\item If $1\leqslant p,q,r,\leqslant \infty$ are such that $\frac{1}{p}+\frac{1}{q}=\frac{1}{r}$, then \[\|(x_n\otimes y_n)_{n=1}^\infty\|_{\ee,r}^w \leqslant \|(x_n)_{n=1}^\infty\|_{\ee,p}^w \|(y_n)_{n=1}^\infty\|_{\ee,q}^w.\] \item For $1\leqslant p\leqslant \infty$, \[\|(x_n\otimes y_n)_{n=1}^\infty\|_{\ee,p}^w \leqslant \|(x_n)_{n=1}^\infty\|_\infty \|(y_n)_{n=1}^\infty\|_p^w.\] \item For $1\leqslant p,q\leqslant \infty$ with $1/p+1/q=1$, \[\|(x_n\otimes y_n\otimes z_n)_{i=1}^\infty\|_{\ee,1}^w \leqslant \|(x_n)_{n=1}^\infty\|_\infty \|(y_n)_{i=1}^\infty\|_p^w \|(z_n)_{n=1}^\infty\|_q^w.\]   \item \[\|(x_n\otimes y_n)_{n=1}^\infty\|_{\pi,1}^w \leqslant \|(x_n)_{n=1}^\infty\|_1^w \|(y_n)_{i=1}^\infty\|_1^w.\] \item If $Y^*$ has cotype $q$, \[\|(x_n\otimes y_n)_{n=1}^\infty\|_{\pi,q}^w \leqslant C_q(Y^*)\|(x_n)_{n=1}^\infty\|_1^w \|(y_n)_{n=1}^\infty\|_\infty.\] \end{enumerate}

Moreover, all of the preceding inequalities hold for finite sequences as well. 

\label{pupper}
\end{proposition}

\begin{proof}$(i)$ Fix $x^*\in B_{X^*}$ and $y^*\in B_{Y^*}$ and note that, since $1=\frac{r}{p}+\frac{r}{q}$,  \begin{align*} \Bigl(\sum_{n=1}^\infty |\langle x^*\otimes y^*, x_n\otimes y_n\rangle|^r\Bigr)^{1/r} & = \Bigl(\sum_{n=1}^\infty |\langle x^*, x_n\rangle|^r |\langle y^*,  y_n\rangle|^r\Bigr)^{1/r} \\ & \leqslant \Bigl[\Bigl(\sum_{n=1}^\infty |\langle x^*, x_n\rangle|^{r(p/r)}\Bigr)^{r/p} \Bigl(\sum_{n=1}^\infty |\langle y^*, y_n\rangle|^{r(q/r)}\Bigr)^{r/q}\Bigr]^{1/r} \\ & = \Bigl(\sum_{n=1}^\infty|\langle x^*, x_n\rangle|^p\Bigr)^{1/p}\Bigl(\sum_{n=1}^\infty |\langle y^*, y_n\rangle|^q\Bigr)^{1/q} \\ & \leqslant \|(x_n)_{n=1}^\infty \|_{p}^w \|(y_n)_{n=1}^\infty\|_{q}^w. \end{align*} 

Since $K=\{x^*\otimes y^*: x^*\in B_{X^*}, y^*\in B_{Y^*}\}$ is a $1$-norming set for $X\tim_\ee Y$, we are done.

$(ii)$ Apply $(i)$ with $r=p$ and $q=\infty$.

$(iii)$ Apply $(ii)$ and then $(i)$ with $r=1$ to obtain \[\|(x_n\otimes y_n\otimes z_n)_{n=1}^\infty\|_{\ee,1}^w \leqslant \|(x_n)_{n=1}^\infty\|_\infty\|(y_n\otimes z_n)_{n=1}^\infty\|_{\ee,1}^w \leqslant \|(x_n)_{n=1}^\infty\|_\infty \|(y_n)_{n=1}^\infty\|_p^w\ \|(z_n)_{n=1}^\infty\|_q^w.\]

$(iv)$ Recall that $(X\tim_\pi Y)^*=\mathfrak{L}(X,Y^*)$. Fix $m\in\nn$ and $T:X\to Y^*$ with $\|T\|=1$. Fix a sequence $(\ee_n)_{n=1}^m$ of iid Rademacher random variables.  For each $1\leqslant n\leqslant m$, let $\sigma_n$ be such that $|\sigma_n|=1$ and $|\langle Tx_n, y_n\rangle|=\sigma_n \langle Tx_n, y_n\rangle$.   Then  \begin{align*}\sum_{n=1}^m|\langle T, x_n\otimes y_n\rangle| & = \sum_{n=1}^m |\langle Tx_n,y_n\rangle| =\sum_{n=1}^m \sigma_n \langle Tx_n, y_n\rangle = \mathbb{E}\Bigl\langle T\sum_{n=1}^m \ee_n \sigma_n x_n, \sum_{n=1}^m \ee_ny_n\Bigr\rangle \\ & \leqslant \mathbb{E}\Bigl\|\sum_{n=1}^m \ee_n\sigma_n x_n\Bigr\|\Bigl\|\sum_{n=1}^m \ee_n y_n\Bigr\| \leqslant \|(x_n)_{n=1}^\infty\|_1^w \|(y_n)_{n=1}^\infty\|_1^w.\end{align*} Since $m\in\nn$ was arbitrary, we are done.

$(v)$ Fix $m\in\nn$ and $T:X\to Y^*$ with $\|T\|=1$.   Fix a sequence $(\ee_n)_{n=1}^m$ of iid Rademacher random variables.   Then \begin{align*}\Bigl(\sum_{n=1}^m \|Tx_n\|^q\Bigr)^{1/q} & \leqslant C_q(Y^*) \Bigl(\mathbb{E}\Bigl\|T\sum_{n=1}^m \ee_n x_n\Bigr\|^q\Bigr)^{1/q} \leqslant C_q(Y^*) \Bigl(\mathbb{E}\Bigl\|\sum_{n=1}^m \ee_n x_n\Bigr\|^q\Bigr)^{1/q} \leqslant C_q(Y^*)\|(x_n)_{n=1}^\infty\|_1^w.\end{align*} Since this holds for any $m\in\nn$, $\Bigl(\sum_{n=1}^\infty \|Tx_n\|^q\Bigr)^{1/q} \leqslant C_q(Y^*)\|(x_n)_{n=1}^\infty\|_1^w$.   Therefore \begin{align*} \Bigl(\sum_{n=1}^\infty |\langle T, x_n\otimes y_n\rangle|^q\Bigr)^{1/q} & = \Bigl(\sum_{n=1}^\infty |\langle Tx_n, y_n\rangle|^q\Bigr)^{1/q} \leqslant \|(y_n)_{n=1}^\infty\|_\infty \Bigl(\sum_{n=1}^\infty \|Tx_n\|^q\Bigr)^{1/q} \\ & \leqslant C_q(Y^*)\|(x_n)_{n=1}^\infty\|_1^w \|(y_n)_{n=1}^\infty\|_\infty. \end{align*}

We can extend these inequalities to finite sequences by extending finite sequences by zeros.

\end{proof}

Let us note an easy consequence of Proposition \ref{sm}. 

\begin{corollary} If $X$ is a Banach space such that $X^*, Y^*$ have the Schur property and $X^*$ has the approximation property, then \[(X\tim_\pi Y)^*=X^*\tim_\ee Y^*.\] 

\label{t1}
\end{corollary}

\begin{proof} Pe\l czy\'{n}ski \cite{P} showed that if a Banach space $X$ admits a subspace isomorphic to $\ell_1$, then $X^*$ contains a copy of $L_1[0,1]$, and therefore $X^*$ lacks the Schur property. By contraposition, if $X^*$ has the Schur property, then $X$ does not contain an isomorph of $\ell_1$.    By Proposition \ref{sm}, \[(X\tim_\pi Y)^*=\mathfrak{L}(X,Y^*)=\mathfrak{K}(X,Y^*)=X^*\tim_\ee Y^*.\]

\end{proof}

\begin{corollary} If $X^*$ has the Schur and approximation properties, then for all $N\in\nn$, $(\tim_\pi^N X)^*=\tim_\ee^N X^*$ and $(\tim_{\pi,s}^NX)^*=\tim_{\ee,s}^N X^*$. 
\label{easy1}
\end{corollary}

\begin{proof} We prove the result by induction on $N$. The $N=1$ case is true by convention and deducing the $N+1$ case assuming the $N$ case follows from Corollary \ref{t1} applied with $Y=\tim_\pi^N X$. Here we note that by the inductive hypothesis, $Y^*=\tim_\ee^N X^*$, which has the Schur property by Proposition \ref{sm}. 

The statement about symmetric tensor products follows from the statement about full tensor products using Lemma \ref{l1}.

\end{proof}

\begin{corollary} If $X$ is any Banach space such that $X^*$ is isomorphic to $\ell_1(\Gamma)$ for some set $\Gamma$, then for all $N\in\nn$, $(\tim_\pi^N X)^*=\tim_\ee^N X^*$ and $(\tim_{\pi,s}^N X)^*=\tim_{\ee,s}^N X^*$.  
\label{easy2}
\end{corollary}

\begin{proof} Since $\ell_1(\Gamma)$ has the approximation and Schur properties, Corollary \ref{easy1} applies. 

\end{proof}

We recall that a topological space is called \emph{scattered} if every non-empty subset has an isolated point. Rudin \cite{R} showed that for any compact, Hausdorff, scattered space, the Dirac functionals on $K$ have dense span in $C(K)^*$, so $C(K)^*=\ell_1(K)$ isometrically. By Corollary \ref{easy2}, we get the following isometric result.

\begin{corollary} Let $K$ be a compact, Hausdorff, scattered topological space. Then for all $N\in\nn$, $(\tim_\pi^N C(K))^*=\tim_\ee^N \ell_1(K)$ and $(\tim_{\pi,s}^N C(K))^*=\tim_{\ee,s}^N \ell_1(K)$.

\end{corollary}

We note that the spaces to which Corollary \ref{easy2} applies have rich subspace structure. In \cite{FOS}, the authors modified the original construction of Bourgain and Delbaen \cite{BD} to show that any Banach space $Z$ such that $Z^*$ is separable is embeddable into a Banach space $X$ such that $X^*$ is isomorphic to $\ell_1$. From this we deduce the following.

\begin{corollary} If $Z$ is any Banach space such that $Z^*$ is separable, then there exists a Banach space $X$ containing an isomorph of $Z$ such that for all $N\in\nn$, $\mathcal{L}_N(X)=\tim_\ee^N X^*$ and $\mathcal{P}_N(X)=\tim_{\ee,s}^N X^*$.

\end{corollary}

We also note that the class of spaces $X$ for which $(\tim_\pi^N X)^*=\tim_\ee^N X^*$ for all $N\in\nn$ is closed under self-tensorization. 

\begin{proposition} If $X$ is a Banach space such that $(\tim_\pi^N X)^*=\tim_\ee^N X^*$ for all $N\in\nn$, then for any $M,N\in\nn$, $(\tim_\pi^N (\tim_\pi^M X))^* = \tim_\ee^N (\tim_\pi^M X)^*$.

\end{proposition}

\begin{proof} For any $M,N\in\nn$, \[(\tim_\pi^N (\tim_\pi^M X))^* = (\tim_\pi^{MN} X)^*= \tim_\ee^{MN}X^* = \tim_\ee^N \tim_\ee^MX^* = \tim_\ee^N(\tim_\pi^M X)^*.\]

\end{proof}

\section{Asymptotic properties}

We first recall the Szlenk index.   For a Banach space $X$, a weak$^*$-compact set $K\subset B_{X^*}$, and $\ee>0$, we let $s_\ee(K)$ denote the subset of $K$ consisting of those $x^*\in K$ such that for any weak$^*$-neighborhood $V$ of $x^*$, $\text{diam}(K\cap V)>\ee$.  We define the transfinite derivatives by \[s^0_\ee(K)=K,\] \[s^{\xi+1}_\ee(K)=s_\ee(s^\xi_\ee(K)),\] and if $\xi$ is a limit ordinal, \[s^\xi_\ee(K)=\bigcap_{\zeta<\xi}s^\zeta_\ee(K).\]    We say $K$ is \emph{weak}$^*$-\emph{fragmentable} provided that for all $\ee>0$, there exists $\xi$ such that $s^\xi_\ee(K)=\varnothing$.     In this case, we let $Sz(K,\ee)=\min \{\xi: s^\xi_\ee(K)=\varnothing\}$ and $Sz(K)=\sup_{\ee>0} Sz(K, \ee)$.  If $A:X\to Y$ is an operator, we let $Sz(A)=Sz(A^*B_{Y^*})$. For a Banach space $X$, we let $Sz(X)=Sz(B_{X^*})=Sz(I_X)$.    We recall the following properties of the Szlenk index of an operator.  If $K$ fails to be weak$^*$-fragmentable, we write $Sz(K)=\infty$.  For an operator $A:X\to Y$, we let $Sz(A)=\infty$ if $Sz(A^*B_{Y^*})=\infty$ and $Sz(X)=\infty$ if $Sz(B_{X^*})=\infty$.    We agree to the convention that $\xi<\infty$ for any ordinal $\xi$.

\begin{proposition}\cite{B} \begin{enumerate}[label=(\roman*)]\item For any Asplund operator $A:X\to Y$, there exists an ordinal $\xi$ such that $Sz(A)=\omega^\xi$. \item For Banach spaces $X,Y,Z$ and operators $A:X\to Y$, $B:Y\to Z$, $Sz(BA) \leqslant \min\{Sz(A), Sz(B)\}$. \item For Banach spaces $X,Y$ and operators $A,B:X\to Y$, $Sz(A+B)\leqslant \max\{Sz(A), Sz(B)\}$. \item For an operator $A:X\to Y$, $Sz(A)=1$ if and only if $A$ is compact. \end{enumerate}

\label{bel}
\end{proposition}

We isolate the following easy consequence of the preceding facts. 

\begin{proposition} Let $X, X_1, Y, Y_1$ be Banach spaces. Let $A:X\to Y$ be an operator and let $P:X_1\to Y_1$ be a finite rank operator operator. Then $Sz(A\otimes P:X\tim_\pi X_1\to Y\tim_\pi Y_1)\leqslant Sz(A)$. 
\label{tns}
\end{proposition}

\begin{proof} If $P=0$, then $A\otimes P=0$ and $Sz(A\otimes P)=1\leqslant Sz(A)$.   Otherwise we can write $P=\sum_{i=1}^n p_i$, where each $p_i$ is a rank $1$ operator.  for each $1\leqslant i\leqslant n$, $A\otimes p_i$ and $A$ factor through each other, so $Sz(A\otimes p_i)=Sz(A)$ by Proposition \ref{bel} $(ii)$.  Therefore \[Sz(A\otimes P)= Sz\bigl(A\otimes\sum_{i=1}^n p_i\bigr) = Sz\bigl(\sum_{i=1}^n A\otimes p_i\bigr) \leqslant \max_{1\leqslant i\leqslant n} Sz(A\otimes p_i)=Sz(A).\]

\end{proof}

Given a set $\Lambda$, we let $\Lambda^{<\omega}$ denote the finite sequences whose members lie in $\Lambda$, including the empty sequence $\varnothing$. We order $\Lambda<\omega$ by initial segments, denoted $\prec$. That is, $s \prec t$ if s is a proper initial
segment of $t$. We write $s\sim t$ to mean that either $s\preceq t$ or $t\prec s$, and we write $s\not\sim t$ to mean that $s\sim t$ and $s\neq t$. For $s\in \Lambda^{<\omega}$, we let $|s|$ denote the length of $s$. For $0\leqslant  i\leqslant |s|$, we let $s|i$ denote the initial segment of $s$ having length $i$. If $|s|>0$, we let $s^-$ denote the maximal, proper initial segment of $s$.  We let $s\smallfrown t$ denote the concatenation of $s$ with $t$.  Given a subset $T$ of $\Lambda^{<\omega}$, we let $MAX(T)$
denote the set of maximal members of $T$ with respect to the order $\preceq$. We define the \emph{derivative} of $T$, denoted $T'$, to be $T'=T\setminus MAX(T)$. We define the transfinite derivatives by \[T^0=T,\] \[T^{\xi+1}=(T^\xi)',\] and if $\xi$ is a limit ordinal, \[T^\xi=\bigcap_{\zeta<\xi}T^\zeta.\]   If there exists an ordinal $\xi$ such that $T^\xi=\varnothing$, then we say $T$ is \emph{well-founded} and let $\text{rank}(T) $ denote the minimum $\xi$ such that $T^\xi=\varnothing$.  If for all $\xi$, $T^\xi\neq \varnothing$, we say $T$ is \emph{ill-founded} and write $\text{rank}(T)=\infty$. For a set $\Lambda$ and $T\subset T^{<\omega}$, we say $T$ is a \emph{rooted tree} if whenever $s\preceq t\in T$, it follows that $s\in T$.    We say $T$ is a \emph{tree} if $\varnothing\notin T$ and whenever $\varnothing\prec s\preceq t\in T$, it follows that $s\in T$.     Given a tree $T$, a Banach space $X$, and a collection $(x_t)_{t\in T}\subset X$, we say $(x_t)_{t\in T}$ is \emph{weakly null} provided that for any ordinal $\zeta$ and any $t\in (\{\varnothing\}\cup T)^{\zeta+1}$, \[0\in \overline{\{x_s: s\in T^\zeta, s^-=t \}}^\text{weak}.\]  We define what it means for a collection $(x_t^*)_{t\in T}\subset X^*$ to be weak$^*$-null analogously.  For $\ee>0$, we say a collection $(x_t)_{t\in T}$ is $\ee$-\emph{large} provided that $\inf_{t\in T}\|x_t\|\geqslant \ee$.  We note that for the $N=0$ case, in all properties below,   weakly null collections $(x_t)_{t\in T}$ with $\text{rank}(T)=1$ can be used interchangeably with weakly null nets $(x_\lambda)$.

We denote the set of finite subsets of $\nn$ by $[\nn]^{<\omega}$.    We denote the set of infinite subsets of $\nn$ by $[\nn]$ and, for $M\in[\nn]$, we let $[M]$ denote the set of infinite subsets of $M$.    By identifying a subset $F$ of $\nn$ with the sequence obtained by listing the members of $F$ in strictly increasing order, we can and do identify $[\nn]^{<\omega}$ with a subset of  $\nn^{<\omega}$.    For $E,F\in[\nn]^{<\omega}$, we let $E<F$ denote the relation that either $E=\varnothing$, $F=\varnothing$, or $\max E<\min F$. It will often be convenient to use the Schreier sets $\mathcal{S}_N$, $N<\omega$, as index sets.   We let $[\nn]^{<\omega}$ denote the set of finite subsets of $\nn$ and recall that  \[\mathcal{S}_0=\{E\in [\nn]^{<\omega}: |E|\leqslant 1\}\] and \[\mathcal{S}_{N+1}= \{\varnothing\}\cup \Bigl\{\bigcup_{n=1}^m E_1\in [\nn]^{<\omega}: E_1<\ldots <E_m, m\leqslant \min E_1, E_n\in \mathcal{S}_N\Bigr\}.\]   Using our identification of $[\nn]^{<\omega}$ with $\nn^{<\omega}$, we can treat $\mathcal{S}_N$ as a rooted tree and $\mathcal{S}_N\setminus \{\varnothing\}$ as a  tree.

   An easy induction on $N$ shows that for any $\varnothing\neq E\in \mathcal{S}_{N+1}$, we can uniquely represent \[E=\bigcup_{n=1}^m E_n,\] where $E_1<\ldots <E_m$, $\varnothing\neq E_n\in \mathcal{S}_N$, $m\leqslant \min E$, and $E_n\in MAX(\mathcal{S}_N)$ for each $1\leqslant n<m$.    We will refer to $\cup_{n=1}^m E_n$ as the $N+1$-\emph{canonical representation} of $E$.    

It is well-known and an easy inductive proof that $\text{rank}(S_N)=\omega^N+1$ and $\text{rank}(S_N\setminus\{\varnothing\})=\omega^N$.  Moreover, for $F\in \mathcal{S}_N$, there exists $m>\max F$ such that $F\cup \{m\}\in \mathcal{S}_N$ if and only if for all $m>\max F$, $F\cup \{m\}\in \mathcal{S}_N$. Moreover, this property is retained by all derivatives $\mathcal{S}_N^\zeta$.       From this it follows that if $X$ is a Banach space and $(x_F)_{F\in \mathcal{S}_N\setminus \{\varnothing\}}\subset X$ is such that $(x_{F\cup\{m\}})_{m>\max F}$ is weakly null for each $F \in \mathcal{S}_N'$, then $(x_F)_{F\in \mathcal{S}_N\setminus \{\varnothing\}}$ is a weakly null collection.

For $N<\omega$ and $1<p\leqslant \infty$, we let $1/p+1/q=1$ and define the property $\mathfrak{A}_{N,p}$.   We say $X$ has property $\mathfrak{A}_{N,p}$ provided that there exists a constant $\alpha$ such that for any $n\in\nn$ and any tree $T$ with $\text{rank}(T)=\omega^N n$ and any weakly null collection $(x_t)_{t\in T}\subset B_X$, there exist $\varnothing=t_0\prec t_1\prec \ldots \prec t_n$ and $u_1, \ldots, u_n$ such that for each $1\leqslant i\leqslant n$, $t_i\in MAX(T^{\omega^N(n-i)})$,  $u_i\in \text{co}(x_t: t_{i-1}\prec t\preceq t_i)$, and  \[\|(u_i)_{i=1}^n\|_q^w \leqslant \alpha.\]

For $N=0$, the property $\mathfrak{A}_{0,p}$ admits an alternative description in terms of games. In the $N=0$ case, if $T$ is a tree of rank $n$ and if $t_0=\varnothing\prec t_1\prec \ldots \prec t_n$, then for any $1\leqslant i\leqslant n$, $\{t\in T: t_{i-1}\prec t\preceq t_i\}=\{t_i\}$, so in this case there are no convex combinations.    Thus the $N=0$ case is significantly simpler than the $N>0$ cases.  For a fixed $\alpha>0$ and $n\in\nn$, Player I chooses a finite codimensional subspace $X_1$ of $X$ and Player II chooses $x_1\in B_{X_1}$. Player I chooses a finite codimensional subspace $X_2$ of $X$, and Player II chooses $x_2\in B_{X_2}$. Play continues in this way until $x_1, \ldots, x_n$ have been chosen. We refer to this as the $A_p(\alpha,n)$ \emph{game on} $X$. We say Player I wins provided $\|(x_i)_{i=1}^n\|_q^w\leqslant \alpha$, and Player II wins otherwise.  Then $X$ has $\mathfrak{A}_{0,p}$ if and only if for some $\alpha>0$ and all $n\in\nn$, Player I has a winning strategy in this game with this choice of $\alpha$.  The space $X$ has $\mathfrak{A}_{0,p}$ if and only if there exists a constant $\alpha$ such that for every $(e_i)_{i=1}^n\in \{X\}_n$, the $n^{th}$ asymptotic structure of $X$, $\|(e_i)_{i=1}^n\|_q^w\leqslant \alpha$.  Let us also define a second game. For $n\in\nn$ and $\alpha>0$, Player I chooses a finite codimensional subspace $X_1$ of $X$, Player II chooses a norm compact subset $C_1$ of $B_{X_1}$, Player I chooses a finite codimensional subspace $X_2$ of $X$, Player II chooses a norm compact subset $C_2$ of $B_{X_2}$, etc. Play continues until $C_1, \ldots, C_n$ are chosen.  Player I wins if for every $(x_i)_{i=1}^n\in \prod_{i=1}^n C_i$, $\|(x_i)_{i=1}^n\|_q^w \leqslant \alpha$.   Then $X$ has $\mathfrak{A}_{0,p}$ if and only if there exists $\alpha>0$ such that for all $n\in\nn$, Player I has a winning strategy in the second game, which we call the \emph{compact} $A_p(\alpha,n)$ \emph{game on} $X$.

We next define an $\omega^N$ \emph{leveled tree}.  Suppose $T$ is an ill-founded tree with subsets $\Lambda_1, \Lambda_2, \ldots$ (called the \emph{levels} of $T$) such that $\Lambda_1$ is a well-founded tree of rank $\omega^N$ and for each $n\in\nn$, if we define \[T_t=\{s\in \Lambda_{n+1}: t\prec s\}\] for each $t\in MAX(\Lambda_n)$, then $(T_t)_{t\in MAX(\Lambda_n)}$ is a partition of $\Lambda_{n+1}$ and $\text{rank}(T_t)=\omega^N$.  For convenience, we let $\Lambda_0=\{\varnothing\}$, so \[T_\varnothing=\{t\in \Lambda_1: \varnothing\prec t\}=\Lambda_1,\] and the singleton $(T_\varnothing)$ is a trivial partition of $\Lambda_1=T_\varnothing$ into one tree of rank $\omega^N$.   For an $\omega^N$ leveled tree $T$, a Banach space $X$, and a collection $(x_t)_{t\in T}\subset X$, we say $(x_t)_{t\in T}$ is \emph{weakly null} provided that for every $n\in\nn$ and $t\in MAX(\Lambda_{n-1})$, $(x_s)_{s\in T_t}$ is weakly null.

For $1<p\leqslant \infty$, if $1/p+1/q=1$, we say $X$ has $\mathfrak{T}_{N,p}$ provided that there exists a constant $\alpha$ such that for any $\omega^N$ leveled tree $T$, and any weakly null $(x_t)_{t\in T}\subset B_X$, there exist $\varnothing=t_0\prec t_1\prec \ldots$ such that $t_n\in MAX(\Lambda_n)$  and $u_n\in \text{co}(x_t: t_{n-1}\prec t\preceq t_n)$ such that $\|(u_n)_{n=1}^\infty\|_q^w\leqslant \alpha$.  Again, in the $N=0$ case, this property admits a simpler description.   For a fixed $\alpha>0$, we can define a two-player game where Player I chooses a finite codimensional subspace $X_1$ of $X$, Player II chooses $x_1\in B_{X_1}$, Player I chooses a finite codimensional subspace $X_2$ of $X$, Player II chooses $x_2\in B_{X_2}$, etc. We refer this as the $T(\alpha,p)$ \emph{game on} $X$   Player I wins if $\|(x_n)_{n=1}^\infty\|_q^w \leqslant \alpha$.    Then $X$ has $\mathfrak{T}_{0,p}$ if there exists $\alpha>0$ such that Player I has a winning strategy in this game.   We also note that this is easily seen to be equivalent to Player I having a winning strategy for some $\alpha>0$ in the \emph{compact} $T(\alpha,p)$ \emph{game on } $X$,  where on the $n^{th}$ turn, Player II chooses a norm compact $C_n\subset B_{X_n}$ rather than simply choosing $x_n\in B_{X_n}$.

The property $\mathfrak{T}_{N,p}$ is also related to the higher moduli of asymptotic uniform smoothness.    For $x\in X$, a tree of rank $\omega^N$, and a weakly null collection $(x_t)_{t\in T}\subset B_X$, we let \[\varrho^X_N(\varsigma,x, (x_t)_{t\in T})=\inf\{\|x+\varsigma y\|-1: t\in T, y\in \text{co}(x_s: s\preceq t)\}.\]    We define \[\varrho^X_N(\varsigma, x)= \sup \Bigl\{\varrho^X_N(\varsigma, x, (x_t)_{t\in T}): \text{rank}(T)=\omega^N, (x_t)_{t\in T}\subset B_X\text{\ weakly null}\Bigr\}.\] In the case $N=0$, this is equal to \begin{align*} \varrho^X_0(\varsigma, x) & = \sup\Bigl\{\|x+\varsigma x_\lambda\|-1: (x_\lambda)\subset B_X\text{\ is a weakly null net}\Bigr\} \\ & = \underset{\dim(X/E)<\infty}{\inf}\underset{y\in B_E}{\sup} \|x+\varsigma y\|-1.\end{align*} We define $\varrho^X_N(\varsigma)=\sup_{x\in B_X} \varrho^X_N(\varsigma, x)$. For $1<p<\infty$, we say $X$ is $N$-$p$-\emph{asymptotically uniformly smooth} (or $N$-$p$-\emph{AUS}) if $\sup_{\varsigma>0} \varrho^X_N(\varsigma)/\varsigma^p<\infty$.     We note that for $1<p<\infty$, $X$ has $\mathfrak{T}_{N,p}$ if and only if there exists an equivalent norm $|\cdot|$ on $X$ such that $(X, |\cdot|)$ is $N$-$p$-AUS.  We say $X$ is $N$-\emph{asymptotically uniformly flat} (or $N$-\emph{AUF}) if there exists $\varsigma_0>0$ such that $\varrho^X_N(\varsigma_0)=0$.     We note that $X$ has $\mathfrak{T}_{N,\infty}$ if and only if there exists an equivalent norm $|\cdot|$ on $X$ such that $(X, |\cdot|)$ is $N$-AUF.

We collect the following facts regarding the propositions above. 

\begin{proposition} Let $X,Y$ be Banach spaces, $1<p\leqslant \infty$, and $N\in \{0\}\cup \nn$.  \begin{enumerate}[label=(\roman*)]\item For an operator $A:X\to Y$, $Sz(A) \leqslant \omega^N$ if and only if for any tree $T$ with rank $\omega^N$ and any weakly null collection $(x_t)_{t\in T}\subset B_X$, \[\inf\{\|Ax\|: t\in MAX(T), x\in \text{\emph{co}}(x_s: s\preceq t)\}=0.\] Applying this with $Y=X$ and $A=I_X$, $Sz(X)\leqslant \omega^N$ if and only if for any tree $T$ with rank $\omega^N$ and any weakly null collection $(x_t)_{t\in T}\subset B_X$, \[\inf\{\|x\|: t\in MAX(T), x\in \text{\emph{co}}(x_s: s\preceq t)\}=0.\] \item If  $X$ has $\mathfrak{T}_{N,p}$, it has $\mathfrak{A}_{N,p}$, but the converse need not hold. \item If $X$ has $\mathfrak{A}_{N,p}$, then $X$ has $\mathfrak{T}_{N,r}$ for all $1<r<p$, but the converse need not hold.  \item If $X$ has $\mathfrak{A}_{N,p}$, then $Sz(X)\leqslant \omega^{N+1}$.   
\end{enumerate}
\label{fha}
\end{proposition}

\begin{proof}Item $(i)$ is \cite[Theorem $4.6$$(i)$]{C}.   Items $(ii)$, $(iii)$, and $(iv)$ are all contained in \cite[Theorem $1.1$]{C1}.

\end{proof}

We note that any infinite dimensional Banach space which has property $\mathfrak{T}_{0,\infty}$ contains an isomorphic copy of $c_0$, and a separable Banach space has  property $\mathfrak{T}_{0,\infty}$ if and only if it is isomorphic to a subspace of $c_0$. In particular, such a space cannot be reflexive unless it has finite dimension. However, the property $\mathfrak{A}_{0,\infty}$ is not so restrictive. 

We conclude this section with several examples. Our first example is Tsirelson's space. Tsirelson's space is a reflexive Banach space with canonical, $1$-unconditional basis $(t_j)_{j=1}^\infty$ having the property that for any $n\in\nn$,  any $n\leqslant k_0<\ldots < k_n$, and any $x_i\in \text{span}\{t_j: j\in [k_{i-1}, k_i)\}$, \[\Bigl\|\sum_{i=1}^n x_i\Bigr\|\leqslant \max_{1\leqslant i\leqslant n}\|x_i\|.\] It easily follows from this fact that $T$ has property $\mathfrak{A}_{0,\infty}$. 

In \cite{AGM}, the authors produced an example of a Banach space $\mathfrak{X}$, which we call the Argyros-Gasparis-Motakis space, such that $\mathfrak{X}$ has property $\mathfrak{A}_{0,\infty}$, $\mathfrak{X}^*$ is isomorphic to $\ell_1$, and $\mathfrak{X}$ does not contain any isomorph of $c_0$. Since $\mathfrak{X}$ does not contain an isomorph of $c_0$, $\mathfrak{X}$ does not have property $\mathfrak{T}_{0,\infty}$.

Finally, if $K$ is scattered, compact, Hausdorff with finite Cantor-Bendixson index, $C(K)$ has property $\mathfrak{T}_{0,\infty}$. This was shown by Lancien in \cite{L}. In fact, if $K$ has finite Cantor-Bendixson index, one can easily write down the formula for the \emph{Grasberg norm} on $C(K)$ and show that this norm is an AUF norm on $C(K)$. Fix $n\in\nn$ such that $K^0, \ldots, K^{n-1}\neq\varnothing$ and $K^n=\varnothing$. Note that $K^{n-1}$ must be finite, since $K^n=(K^{n-1})'=\varnothing$.  Define \[|f|= \max_{0\leqslant i<n} 2^i\|f|_{K^i}\|,\] which is a norm on $C(K)$ which is  $2^{n-1}$-equivalent to the canonical norm.   Note that for any $0\leqslant i<n$, $\ee>0$, and $f\in C(K)$, \[\{\varpi\in K^i: 2^i|f(\varpi)|\geqslant \ee+|f|/2\}\] is finite. Indeed, if this set were infinite, it would necessarily have a non-isolated point $\varpi\in K^{i+1}$, which would satisfy \[2\ee +|f| \leqslant 2^{i+1}|f(\varpi)| \leqslant |f|,\] a contradiction.    Therefore if $|f|\leqslant 1$, then for any $\ee>0$, the set $L:= \cup_{i=0}^{n-1}\{\varpi\in K^i: 2^i|f(\varpi)| \geqslant \ee+1/2\}$ is finite.    Then for any $g\in C(K)\cap\bigcap_{\varpi\in L}\ker(\delta_\varpi)$ with $|g|\leqslant 1$, $|f+\frac{1}{2}g |\leqslant 1+\ee$. Since this holds for any $\ee>0$ and $f\in C(K)$ with $|f|=1$, it follows that  $\varrho^{(C(K), |\cdot|)}(1/2)=0$.   Let us see the inequality. Fix $0\leqslant i<n$ and $\varpi\in K^i$. If $\varpi\in L$, then $g(\varpi)=0$ and \[2^i|(f+g/2)(\varpi)|=2^i|f(\varpi)|\leqslant 1.\]   If $\varpi\in K^i\setminus L$, then \[2^i(f+g/2)(\varpi)| \leqslant 2^i|f(\varpi)|+ 2^i|g(\varpi)|/2 \leqslant \ee+1/2+|g|/2 \leqslant 1+\ee.\]

\section{The upper estimates} \label{sec: upperestimates}

Let us say a Banach space $X$ has the \emph{tail approximation property} (or \emph{TAP}) provided that there exists a constant $\beta>0$ such that for any finite codimensional subspace $W$ of $X$, there exists a finite codimensional subspace $Y$ of $X$ such that $Y\subset W$ and such that  there exist projections $P^Y_0, P^Y_1$ on $X$ such that $\|P^Y_0\|, \|P^Y_1\|\leqslant \beta$, $I_X=P^T_0+P^Y_1$,  and such that $P^Y_1(X)=Y$. If the property holds with constant $\beta$, we say $X$ has $\beta$-\emph{TAP}. 

We note that $X$ has TAP if and only if $X^*$ is a $\Pi_\lambda$ space, which means that every finite dimensional subspace $E$ of $X^*$ is contained in a finite dimensional subspace $F$ of $X^*$ which is $\lambda$-complemented in $X^*$. To see this equivalence, note that if $X$ has $\beta$-TAP and $E=\text{span}\{x^*_1, \ldots, x_n^*\}$ is a finite dimensional subspace of $X^*$, then we can define $W=\cap_{i=1}^n \ker(x^*_i)$. Since $X$ has $\beta$-TAP, there exists a finite codimensional subspace $Y$ of $X$ such that $Y\subset W$ and such that  there exist projections $P^Y_0, P^Y_1$ on $X$ such that $\|P^Y_0\|, \|P^Y_1\|\leqslant \beta$, $I_X=P^T_0+P^Y_1$,  and such that $P^Y_1(X)=Y$. Then $F:=(P_0^Y)^*(X^*)\supset E$ is $\beta$-complemented.     The converse follows from adapting the arguments of \cite{D}, wherein it is shown using the principle of local reflexivity that every finite rank projection on $X^*$ is close to the adjoint of a finite rank projection on $X$.   

Let us note that in order to verify that a space $X$ has TAP, it is sufficient to know that there exist $\beta_0\geqslant 1$ and a subset $F\subset X^*$ with dense span in $X^*$ such that for every finite subset $F_0\subset F$, there exists a finite subset $F_0\subset F'\subset F$ such that  $\cap_{\psi\in F'}\ker(\psi)$ is $\beta_0$-complemented in $X$. Indeed, by first choosing an Auerbach basis and then using local reflexivity, for any finite subset $E$ of $X^*$ and $\ee\in (0,1)$, we can assume there exist $(x_i)_{i=1}^m\subset (1+\ee)B_X$ and $(\xi_i)_{i=1}^m\subset S_{X^*}$ such that $E\subset \text{span}\{\xi_i:1\leqslant i\leqslant n\}$ and $\langle \xi_i, x_j\rangle =\delta_{i,j}$ for all $1\leqslant i,j\leqslant n$.      We can then find a finite subset $F'$ of $F$,  a projection $Q:X\to X$ with range $\cap_{\psi\in F'}\ker(\psi)$, and $\psi_1, \ldots, \psi_n\in \text{span}\{F\}$ such that $\|\xi_i-\psi_i\|<\frac{\ee}{(1+\ee)n}$.   We then define $T=\sum_{i=1}^n (\xi_i-\psi_i)\otimes x_i$ and note that $\|T\|<\ee$.   Therefore $J=I-T$ is an isomorphism of $X$ with $\|J^{-1}\|\leqslant (1-\ee)^{-1}$.  Note also that $J^* \xi_i=\psi_i$ for all $1\leqslant i\leqslant n$. Let $G=\{(J^{-1})^*\psi: \psi\in F'\}\supset \{\xi_1, \ldots, \xi_n\}$ and note  that $\cap_{\xi\in G} \ker(\xi)=J(\cap_{\psi\in F'} \ker(\psi))$.   This means $\cap_{\xi\in E}\ker(\xi)\subset \cap_{\xi\in G} \ker(\xi)$ is the range of the projection $JQJ^{-1}$, and $\|JQJ^{-1}\|, \|I-JQJ^{-1}\|\leqslant 1+ \beta_0 \frac{1+\ee}{1-\ee}$.

We will have two examples in mind of spaces with TAP. If $X$ is any Banach space with shrinking basis $(e_i)_{i=1}^\infty$ and coordinate functionals $(e^*_i)_{i=1}^\infty$, then $F=\{e^*_i: i\in\nn\}$ satisfies the conditions of the previous paragraph with $\beta_0-1$ equal to the basis constant of $(e_i)_{i=1}^\infty$.  For any finite subset $F_0=\{e^*_{i_k}: 1\leqslant k\leqslant n\}\subset F$, if  $F'=\{e^*_i: i\leqslant \max_{1\leqslant k\leqslant n} i_k\}$, then $\cap_{\psi\in F'}\ker(\psi)$ is $\beta_0$-complemented in $X$.   

Our second example is $C(K)$ spaces with $K$ scattered. If $K$ is scattered, then $C(K)^*=\ell_1(K)$. We let $F=\{\delta_\varpi: \varpi\in K\}$. For any finite subset $F_0=\{\delta_\varpi: \varpi\in L\}$, $L\subset K$ finite, is such that $\cap_{\psi\in F_0}\ker(\psi)$ is $2$-complemented in $C(K)$. Indeed, we can choose for each $\varpi\in L$ a continuous function $g_\varpi:K\to [0,1]$ such that the functions $(g_\varpi)_{\varpi\in L}$ are disjointly supported and $g_\varpi(\varpi)=1$.    Then $I-\sum_{\varpi\in L} \delta_\varpi\otimes g_\varpi$ is a projection onto $\cap_{\psi\in F_0} \ker(\psi)$ with norm at most $2$.

We next state two lemmas which we will need for our main result of this section. 

\begin{lemma} Let $X$ be a Banach space, $\beta\geqslant 1$, and let $P_0,P_1:X\to X$ be projections on $X$ such that $P_0+P_1=I_X$ and $\|P_0\|, \|P_1\|\leqslant \beta$.  Fix $N\in\nn$. For $\eta=(\eta_i)_{i=1}^N\in 2^N:=\{0,1\}^N$, let $P_\eta=\otimes_{i=1}^N P_{\eta_i}$.   Let $H\subset 2^N$, $\ee>0$, and  $u\in B_{\tim_\pi^N X}$, assume $\|\sum_{\eta\in H} P_\eta u\|<\ee/2$.  Then for each $\eta\in 2^N$, there exist finite sets $F^\eta_1, \ldots, F^\eta_N\subset B_X$ and a tensor \[v_\eta\in \text{\emph{co}}\{\otimes_{i=1}^N x_i: (x_i)_{i=1}^N\in \prod_{i=1}^N F^\eta_i\}\] such that if for some $1\leqslant i\leqslant n$, if $\eta(i)=1$, then $F^\eta_i \subset P_1(X)$, and such that \[\Bigl\|u-\beta^N\sum_{\eta\in 2^N\setminus H} v_\eta\Bigr\| <\ee.\]  
\label{ted}
\end{lemma}

\begin{proof} Fix $m\in\nn$, vectors $x_{ij}\in B_X$, $1\leqslant i\leqslant N$, $1\leqslant j\leqslant m$, and positive numbers $(w_j)_{j=1}^m$ summing to $1$ such that, with $v=\sum_{j=1}^m w_j\otimes_{i=1}^N x_{ij}$,  $\|u-v\|<\frac{\ee/2}{2^N\beta^N}$.    For each $\eta\in 2^N$, let \[F^\eta_i= \{\beta^{-1} P_{\eta(i)} x_{ij}: 1\leqslant j\leqslant m\}\subset B_X.\] Note that if $\eta(i)=1$, then $F^\eta_i\subset P_1(X)$.  Let \[v_\eta = \sum_{j=1}^m w_j \otimes_{i=1}^N \beta^{-1} P_{\eta(i)}x_{ij} \in \text{co}\{\otimes_{i=1}^N x_i: (x_i)_{i=1}^N\in \prod_{i=1}^N F^\eta_i\}.\]   Finally, we note that \begin{align*} \Bigl\|u-\beta^N\sum_{\eta\in 2^N \setminus H} v_\eta\Bigr\| & = \Bigl\|u-\sum_{\eta\in 2^N\setminus H} P_\eta v\Bigr\| \leqslant \Bigl\|\sum_{\eta\in H} P_\eta u\Bigr\|+ \sum_{\eta\in 2^N\setminus H} \|P_\eta\|\|u-v\| \\ & <\ee/2 + \frac{\ee/2}{2^N\beta^N}\sum_{\eta\in 2^N\setminus H}\beta^N\leqslant \ee. \end{align*} 

\end{proof}

\begin{lemma} Let $X$ be a Banach space and fix $N,n\in\nn$.    Assume $\alpha>0$ and $F_1, \ldots, F_n \subset B_X$ are non-empty sets such that for any $(x_i)_{i=1}^n\in \prod_{i=1}^n F_i$, $\|(x_i)_{i=1}^n\|_1^w\leqslant \alpha$.    \begin{enumerate}[label=(\roman*)]\item If $x_{ij}\in F_i$ for each $1\leqslant i\leqslant n$ and $1\leqslant j\leqslant N$, then \[\|(\otimes_{j=1}^N x_{ij})_{i=1}^n\|_1^w \leqslant \alpha^N.\] \item If for each $1\leqslant i\leqslant n$, \[v_i\in \text{\emph{co}}\{\otimes_{j=1}^N x_{ij}: (x_{ij})_{j=1}^N \in F_i^N\},\] then $\|(v_i)_{i=1}^n\|_1^w \leqslant \alpha^N$.    \item Assume $N>1$. If $X^*$ has cotype $q$ and for some $1\leqslant k\leqslant N$, there exist sets $F_{ij}\subset B_X$ such that for each $1\leqslant i\leqslant n$ and each $1\leqslant j\leqslant n$ with $j\neq k$, $F_{ij}\subset F_i$.   Then if $x_{ij}\in F_{ij}$ for each $1\leqslant i\leqslant n$ and $1\leqslant j\leqslant N$, $\|(\otimes_{j=1}^N x_{ij})_{i=1}^n\|_q^w \leqslant C_q(X^*)\alpha^{N-1}$. \item Assume $N>1$. If $X^*$ has cotype $q$ and for some $1\leqslant k\leqslant N$, there exist sets $F_{ij}\subset B_X$ such that for each $1\leqslant i\leqslant n$ and each $1\leqslant j\leqslant n$ with $j\neq k$, $F_{ij}\subset F_i$.   Then if    \[v_i\in \text{\emph{co}}\{\otimes_{j=1}^N x_{ij}: (x_j)_{j=1}^N\in \prod_{j=1}^N F_{ij}\}\] for each $1\leqslant i\leqslant n$, $\|(v_i)_{i=1}^n\|_q^w \leqslant C_q(X^*) \alpha^{N-1}$. \end{enumerate}

Moreover, the analogous statement holds if we begin with an infinite sequence $F_1, F_2, \ldots \subset B_X$ rather than a sequence of length $n$. 

\label{ted2}
\end{lemma}

\begin{proof}$(i)$ By Proposition \ref{pupper}$(iv)$, $\|(\otimes_{j=1}^N x_{ij})_{i=1}^n\|_1^w \leqslant \prod_{j=1}^N \|(x_{ij})_{i=1}^n\|_1^w$. By hypothesis, the latter product does not exceed $\alpha^N$. 

$(ii)$ The hypotheses yield that for any $(\ee_i)_{i=1}^n$ with $\max_i |\ee_i|\leqslant 1$, \[\sum_{i=1}^n \ee_i v_i \in \text{co}\Bigl\{\sum_{i=1}^n \ee_i \otimes_{j=1}^N x_{ij}: (x_{ij})_{j=1}^N\in F_i^N\Bigr\}, \] so by the triangle inequality and $(i)$, \[\Bigl\|\sum_{i=1}^n \ee_i v_i\Bigr\|\leqslant \sup\Bigl\{\Bigl\|\sum_{i=1}^n \ee_i\otimes_{j=1}^N x_{ij}\Bigr\|: (x_{ij})_{j=1}^N \in F_i^N\Bigr\} \leqslant \alpha^N.\]

$(iii)$ Without loss of generality, we can assume $k=N$.   By Proposition \ref{pupper}$(v)$, \[\|(\otimes_{j=1}^N x_{ij})_{i=1}^n\|_q^w \leqslant C_q(X^*) \|(\otimes_{j=1}^{N-1} x_{ij})_{i=1}^n\|_1^w \|(x_{i N})_{i=1}^n\|_\infty \leqslant C_q(X^*) \alpha^{N-1} \cdot 1=C_q(X^*) \alpha^{N-1}.\] 

$(iv)$ We can deduce $(iv)$ from $(iii)$ the same way $(ii)$ was deduced from $(i)$.

For the moreover statement in the case of $(i)$, we note that for any $x_{ij}\in F_i$, $i\in\nn$, $1\leqslant j\leqslant N$, \[\|(\otimes_{j=1}^N x_{ij})_{i=1}^\infty\|_1^w = \sup_n \|(\otimes_{j=1}^N x_{ij})_{i=1}^n \|_1^w \leqslant \alpha^N\] by the finite case. Items $(ii)$-$(iv)$ can be shown similarly. 

\end{proof}

Before we state our main theorem, we make precise the notion of a winning strategy for Player I in the compact $A(\alpha, n,p)$ and $T(\alpha,p)$ games. Let $\mathcal{K}_X$ denote the set of norm compact subsets of $B_X$. We let $\mathcal{K}_X^{<\omega}$ denote the set of all finite sequences of members of $\mathcal{K}_X$, including the empty sequence, denoted $\varnothing$.  For $n\in\nn$, we let $\mathcal{K}_X^{\leqslant n}$ denote the subset of $\mathcal{K}_X^{<\omega}$ consisting of sequences with length not exceeding $n$. A \emph{strategy for Player I in the compact} $A_p(\alpha,n)$ \emph{game} is a function $\psi:\mathcal{K}_X^{\leqslant n}\to \text{codim}(X)$. A strategy for Player I in the $A_p(\alpha, n)$ game is called a \emph{winning strategy for Player I in the compact} $A_p(\alpha,n)$ \emph{game}  if whenever $(F_i)_{i=1}^n\in \mathcal{K}_X^n$ satisfies $F_i\subset \psi((F_j)_{j=1}^{i-1})$ for each $1\leqslant i\leqslant n$, then $\|(x_i)_{i=1}^n\|_q^w \leqslant \alpha$ for all $(x_i)_{i=1}^n\in \prod_{i=1}^n$.  Here, $(F_j)_{j=1}^0=\varnothing$.   Similarly, a \emph{strategy for Player I in the compact} $T_p(\alpha)$ \emph{game} is a function $\psi:\mathcal{K}_X^{<\omega}\to \text{codim}(X)$. A strategy for Player I in the $T_p(\alpha)$ game is a \emph{winning strategy for Player I in the compact} $T_p(\alpha)$ \emph{game} provided that for any $(F_i)_{i=1}^\infty \subset \mathcal{K}_X$ such that $F_i\subset \psi((F_j)_{j=1}^{i-1})$ for all $i\in\nn$, then $\|(x_i)_{i=1}^\infty\|_q^w \leqslant \alpha$ for any $(x_i)_{i=1}^\infty\in \prod_{i=1}^\infty F_i$.

\begin{theorem} Fix $2\leqslant q<\infty$ and let $1/p+1/q=1$.   For $N\in\nn$, if $X_1, \ldots, X_N$ are Banach spaces each of which has TAP and $\mathfrak{A}_{0,\infty}$ (resp. $T_{0,\infty}$) and such that $X^*_i$ has cotype $q$ for each $1\leqslant i\leqslant N$, then $\tim_{\pi,i=1}^N X_i$ has property $\mathfrak{A}_{N/2-1,p}$ (resp. $\mathfrak{T}_{N/2-1,p}$) if $N$ is even and property $\mathfrak{A}_{(N-1)/2,\infty}$ if $N$ is odd. Moreover, $Sz(\tim_{\pi,i=1}^N X_i)\leqslant \omega^{\lceil N/2\rceil}$. 

\label{upper1}
\end{theorem}

\begin{proof} The last sentence of the theorem follows from the preceding statements of the theorem and Proposition \ref{fha}$(iv)$.    

We also note that if $X_1, \ldots, X_N$ have TAP and $\mathfrak{A}_{0,\infty}$ (resp. $\mathcal{T}_{0,\infty}$) and each $X^*_i$ has cotype $q$, then $X:=\bigl(\oplus_{i=1}^N X_i\bigr)_{\ell_\infty^N}$ has TAP, $\mathfrak{A}_{0,\infty}$ (resp. $\mathcal{T}_{0,\infty}$), and $X^*$ has cotype $q$. Moreover, $\tim_{\pi,i=1}^N X_i$ is isometrically isomorphic to a $1$-complemented subspace of $\tim_\pi^N X$. Therefore we can and do assume $X_1=\ldots =X_N=X$ in the proof.  

We prove by induction on $N$ that if $X$ has TAP and $\mathfrak{A}_{0,\infty}$ (resp. $\mathcal{T}_{0,\infty}$) and $X^*$ has cotype $q$, then $\tim_\pi^{2N-1} X$ has $\mathfrak{A}_{N-1,\infty}$ and $\tim_\pi^{2N} X$ has $\mathfrak{A}_{N-1,p}$. As noted above, these imply that $Sz(\tim_\pi^{2N-1}X), Sz(\tim_\pi^{2N} X) \leqslant \omega^N$, which will also be part of the induction.

For the $N=1$ case, $\tim_\pi^1 X=X$ has $\mathfrak{A}_{0,\infty}$ (resp. $\mathfrak{T}_{0,\infty}$) by hypothesis. 

We next deduce the even cases from previous cases.   Fix $N\in\nn$ and assume the result holds for all positive integers less than $2N$.  If $N>1$, then $Sz(\tim_\pi^{2N-2} X) \leqslant \omega^{N-1}$ by the inductive hypothesis.       Note also that  Fix $\alpha>0$ such that for all $n\in\nn$, Player I has a winning strategy in the compact $A_\infty(\alpha,n)$ game. Fix $\beta\geqslant 1$ such that $X$ has $\beta$-TAP.  Recall that $2^{2N}=\{0,1\}^{2N}$. For $(\eta_i)_{i=1}^{2N}\in \{0,1\}^{2N}$, let $|\eta|=\{i\leqslant 2N: \eta_i=1\}$.    For $W\in \text{codim}(X)$, we let $P^W_0, P^W_1$ denote the projections guaranteed to exist by TAP. Recall that for $\eta=(\eta_i)_{i=1}^{2N}$ and $W\in \text{codim}(X)$, $P^W_\eta=\otimes_{i=1}^{2N} P^W_{\eta_i}$.    Note that if $N=1$, then for any $W\in \text{codim}(X)$ and any $\eta\in 2^{2N}$ with $|\eta|\leqslant 2(N-1)$, $Sz(P^W_\eta)\leqslant \omega^{N-1}$. For $N=1$, this is true because if $|\eta|\leqslant 2(N-1)$, then $|\eta|=0$ and $P_\eta^W= \otimes^2 P^W_0$ is finite rank.    For $N>1$, this claim is true by the inductive hypothesis combined with Proposition \ref{bel}$(ii)$ and Proposition \ref{tns}. Indeed, if $\nu=(1,1,\ldots, 1, 0, 0, \ldots, 0)$ with $|\nu|=m \leqslant 2(N-1)$, then $P_\nu^W=(\otimes^m P^W_1)\otimes (\otimes^{2N-m} P^W_0):(\tim_\pi^m X)\tim_\pi (\tim_\pi^{2N-m} X)\to (\tim_\pi^m X)\tim_\pi (\tim_\pi^{2N-m} X)$. By Proposition \ref{bel}$(ii)$ and the inductive hypothesis, since $m\leqslant 2(N-1)$,  \[Sz(\otimes^m P^W_1)\leqslant Sz(\tim_\pi^m X) \leqslant \omega^{N-1}.\] Since $\otimes^{2N-m} P^W_0$ has finite rank, $Sz(P^W_\nu)\leqslant \omega^{N-1}$ by Proposition \ref{tns}.  Therefore for any $\eta$ with $|\eta|=|\nu|=m$, since $P^W_\eta$, $P^W_{\nu}$ factor through each other, $Sz(P^W_\eta)\leqslant \omega^{N-1}$ for any $\eta\in 2^{2N}$ with $|\eta|<2(N-1)$.      Let $H=\{\eta\in 2^{2N}: |\eta|\leqslant 2(N-1)\}$.    By Proposition \ref{bel}$(iii)$ and the previous claim, for any $W\in \text{codim}(X)$,  \[Sz\Bigl(\sum_{\eta\in H} P_\eta^X\Bigr) \leqslant \max_{\eta\in H} Sz(P^\eta_X)\leqslant \omega^{N-1}.\]    For each $1\leqslant k\leqslant 2N$, let $\eta_k=(1,1,\ldots, 1,0,1,\ldots, 1)$, where the zero occurs at the $k^{th}$ coordinate.  Let $\eta_0=(1,1,\ldots, 1)$.   
	
	Fix $n\in\nn$ and let $\psi:\mathcal{K}_X^{\leqslant n}\to \text{codim}(X)$ be a winning strategy for Player I in the $A_\infty(\alpha,n)$ game.    Fix  $\ee>0$,  $(\ee_i)_{i=1}^\infty \subset (0,1)$ such that $\sum_{i=1}^\infty \ee_i<\ee$.      Let $T$ be a tree with rank $\omega^{N-1} n$ and let $(x_t)_{t\in T}\subset B_{\tim_\pi^{2N} X}$ be weakly null. Let $t_0=\varnothing$.  Let $W_1=\psi(\varnothing)$.  Since $X$ has TAP, there exists $Z_1\leqslant W_1$ such that $\dim(X/Z_1)<\infty$ and such that the projections $P^{Z_1}_0, P^{Z_1}_1:X\to X$ such that $P_0^{Z_1}+P_1^{Z_1}=I_X$, $P_1^{Z_1}(X)=Z_1$, and $\|P^{Z_1}_0\|, \|P^{Z_1}_1\|\leqslant \beta$.   Note that  $T^{\omega^{N-1}(n-1)}$ has rank $\omega^{N-1}$ and $(x_t)_{t\in T^{\omega^{N-1}(n-1)}}\subset B_X$ is weakly null. Since $Sz(\sum_{\eta\in H} P^{Z_1}_\eta) \leqslant \omega^{N-1}$, by Proposition \ref{fha}$(i)$, there exist $t_1\in MAX(\omega^{N-1}(n-1))$ and $u_1\in \text{co}(x_t: t_0\prec t\preceq t_1)$ such that $\|\sum_{\eta\in H}P^{Z_1}_\eta u_1\|<\ee_1/2$.         By Lemma \ref{ted}, there exist finite sets $F^\eta_{1,j}\subset B_{\tim_\pi^{2N}X}$  and \[v^\eta_1 \in \text{co}\Bigl\{\otimes_{j=1}^{2N} x_j: (x_j)_{j=1}^{2N} \in \prod_{j=1}^{2N} F^\eta_{1,j}\Bigr\}\] such that for each $\eta\in 2^{2N}$ and $1\leqslant j\leqslant 2N$ such that $\eta(i)=1$, $F^\eta_{1, j} \subset P^{Z_1}_1(X) =Z_1$, and \[\Bigl\|u_1-\beta^{2N}\sum_{k=0}^{2N} v^{\eta_k}_1\Bigr\|\leqslant \ee_1.\]    Let \[F_1=\Bigl(\bigcup_{j=1}^{2N} F^{\eta_0}_{1,j}\Bigr)\cup \bigcup_{k=1}^{2N}\Bigl(\bigcup_{k\neq j=1}^{2N} F^{\eta_k}_{1,j}\Bigr) \subset B_{Z_1}.\] Here we are using the fact that for each $1\leqslant j\leqslant 2N$, $\eta_0(j)=1$, so $F^{\eta_0}_{1,j}\subset B_X\cap Z_1$ for all $1\leqslant j\leqslant 2N$. Similarly, for each $1\leqslant k\leqslant 2N$ and each $1\leqslant j\leqslant 2N$ with $j\neq k$, $F^{\eta_k}_{1,j}\subset B_X\cap Z_1$.     This completes the first recursive step.

Now assume that for $r<n$, $F_1, \ldots, F_r$, $\varnothing=t_0\prec \ldots \prec t_r$ have been chosen and $t_r\in MAX(T^{\omega^{N-1}(n-r)})$.  Let $W_{r+1}=\psi(F_1, \ldots, F_r)$ and let $Z_{r+1}\subset W_{r+1}$, $P^{Z_{r+1}}_0, P^{Z_{r+1}}$ be as in the definition of TAP.     We note that, with $S=\{t\in T^{\omega^{N-1}(n-r-1)}: t_r\prec t\}$, then $\text{rank}(S)=\omega^{N-1}$ and $(x_t)_{t\in S}$ is weakly null.    Therefore by another application of Proposition \ref{fha}$(i)$, there exist $t_{r+1}\in MAX(S)=MAX(T^{\omega^{N-1}(n-r-1)})$ and \[u_{r+1}\in \text{co}\{x_t:t\in S, t\preceq t_{r+1}\}=\text{co}\{x_t: t_{r-1}\prec t \preceq t_{r+1}\}\] such that $\|\sum_{\eta\in H}P^{Z_{r+1}}_\eta u_{r+1}\|<\ee_{r+1}/2$.    By another application of Lemma \ref{ted}, there exist finite sets $F^\eta_{r+1, j}\subset B_X$ and \[v^\eta_{r+1} \in \text{co}\Bigl\{\otimes_{j=1}^{2N} x_j: (x_j)_{j=1}^{2N} \in \prod_{j=1}^{2N} F^\eta_{r+1, j}\Bigr\}\] such that for each $\eta\in 2^{2N}$ and $1\leqslant j\leqslant 2N$ such that $\eta(j)=1$, $F^\eta_{r+1, j} \subset P^{Z_{r+1}}_1(X) =Z_{r+1}$, and \[\Bigl\|u_{r+1}-\beta^{2N}\sum_{k=0}^{2N} v^{\eta_k}_{r+1}\Bigr\|\leqslant \ee_{r+1}.\] Let \[F_{r+1}=\Bigl(\bigcup_{j=1}^{2N} F^{\eta_0}_{r+1,j}\Bigr)\cup \bigcup_{k=1}^{2N}\Bigl(\bigcup_{k\neq j=1}^{2N} F^{\eta_k}_{r+1,j}\Bigr) \subset B_{Z_{r+1}}.\] The last inclusion is established as in the previous paragraph.  This completes the recursive construction. By Lemma \ref{ted2}$(iv)$, for each $0\leqslant k\leqslant 2N$, $\|(v^{\eta_k}_i)_{i=1}^n\|_q^w \leqslant C_q(X^*)\alpha^{2N-1}$.    Therefore \[\|(u_i)_{i=1}^n\|_q^w \leqslant \beta^{2N}\sum_{k=0}^{2N} \|(v^{\eta_k}_i)_{i=1}^n\|_q^w + \sum_{i=1}^n \Bigl\|u_i-\beta^{2N} \sum_{k=0}^{2N} v^{\eta_k}_i\Bigr\|\leqslant C_q(X^*)\alpha^{2N-1}\beta^{2N}(2N+1) +\ee.\]    Since this constant does not depend on $n$,  we have shown that $\tim_\pi^{2N}X$ has property $\mathfrak{A}_{N-1,p}$ with constant $C_q(X^*)\alpha^{2N-1}\beta^{2N}(2N+1)+\ee$.   If we instead had assumed that $X$ has $\mathfrak{T}_{0,\infty}$, then we could have let $\psi$ be a winning strategy in the $T_\infty(\alpha)$ game. The argument would be essentially the same, except we would use a single, $\omega^{N-1}$ leveled tree $T$ and the recursion would not have terminated after $n$ levels.  This yields the even cases, given the previous cases.

We now prove the statement for $\tim_\pi^{2N+1}X$, assuming the result for $\tim_\pi^{2N}X$.  More precisely, we will prove the statement for $\tim_\pi^{2N+1}X$ assuming  that $Sz(\tim_\pi^{2N}X) \leqslant \omega^N$.      We retain many notations from the previous paragraphs. For any $W\in \text{codim}(X)$ and any $\eta\in 2^{2N+1}$ with $|\eta|\leqslant 2N$, $Sz(P^W_\eta)\leqslant \omega^N$.    The proof of this is that, after permuting factors, $P^W_\eta=(\otimes^{|\eta|} P^W_1)\otimes P$, where $\otimes^{|\eta|}P^W_1$ is an operator on $\tim_\pi^{|\eta|}X$ for some $|\eta|\leqslant 2M$, so its Szlenk index cannot exceed $Sz(\tim_\pi^{2N}X)\leqslant \omega^N$, and $P$ is finite rank.    Let $\eta_0=(1,\ldots, 1)\in 2^{2N+1}$ and let $H=\{\eta\in 2^{2N+1}: |\eta|\leqslant 2N\}$, so $2^{2N+1}\setminus H=\{\eta_0\}$.   Fix $n\in\nn$, let $\psi$ be a winning strategy for Player I in the $A_\infty(\alpha,n)$ game, let $T$ be a tree with rank $\omega^N n$, and let $(x_t)_{t\in T}\subset B_{\tim_\pi^{2N+1} X}$ be weakly null.    We perform the recursion as above, except we deal only with $H$ and are only concerned with $P^{Z_1}_{\eta_0}$.   In this case, at each step, we choose $F^\eta_{i,j}$,  $v^{\eta_0}_i$ as above, except we let \[F_i=\bigcup_{j=1}^{2N+1} F^{\eta_0}_{i,j}\subset B_{Z_i}\] and \[\|u_{i+1}-\beta^{2N+1} v^{\eta_0}_i\|<\ee_i.\]      Since $\eta_0(j)=1$ for all $1\leqslant j\leqslant 2N+1$, the resulting sets $F_1, \ldots, F_n$ are such that $\|(\otimes_{j=1}^{2N+1}x_{ij})_{i=1}^n\|_1^w \leqslant \alpha$ for any $x_{ij}\in F_i$, $1\leqslant i\leqslant n$, $1\leqslant j\leqslant 2N+1$.  By Lemma \ref{ted2}$(ii)$, $\|(v^{\eta_0})_{i=1}^n\|_1^w\leqslant \alpha^{2N+1}$.   Therefore \[\|(u_i)_{i=1}^n\|_1^w \leqslant \beta^{2N+1}\|(v^{\eta_0})_{i=1}^n\|_1^w + \sum_{i=1}^n \|u_i-\beta^{2N+1} v^{\eta_0}v_i\| \leqslant \alpha^{2N+1}\beta^{2N+1}+\ee.\]    This gives that $\tim_\pi^{2N+1}X$ satisfies property $\mathfrak{A}_{N,\infty}$ with constant $\alpha^{2N+1}\beta^{2N+1}+\ee$. Again, the modification in the case that $X$ has $\mathfrak{T}_{N,\infty}$ is straightforward.

\end{proof}

Our next goal will be to show that for any space $X$ satisfying the conditions of Theorem \ref{upper1} and any $N\in\nn$, $\mathfrak{L}(X, \tim_\ee^N X^*)=\mathfrak{K}(X, \tim_\ee^N X^*)$.   For this, we need a substitute for the Schur property. We will use the $1$-Schur property.    

Given a Banach space $Y$, $m\in\nn$,  and a sequence $(y_n)_{n=1}^\infty$, we say $(y_n)_{n=1}^\infty$ is an $\ell_1^m+$-\emph{spreading model} provided that $(y_n)_{n=1}^\infty$ is bounded and \[\inf\Bigl\{\|y\|: \varnothing\neq F\in \mathcal{S}_m, y\in \text{co}\{y_n: n\in F\}\Bigr\} >0.\]    We say that a sequence $(z_n)_{n=1}^\infty\subset Y$ is $m$-\emph{weakly null} provided that for any subsequence $(y_n)_{n=1}^\infty$ of $(z_n)_{n=1}^\infty$ and $\ee>0$, there exist $F\in \mathcal{S}_m$ and positive scalars $(a_n)_{n\in F}$ such that \[\Bigl\|\sum_{n\in F} a_n y_n\Bigr\|<\ee\sum_{n\in F}a_n.\]  Of course, given a weakly null sequence $(z_n)_{n=1}^\infty$, either $(z_n)_{n=1}^\infty$ is $m$-weakly null or $(z_n)_{n=1}^\infty$ has a subsequence which is an $\ell_1^m+$-spreading model. 

   We say an operator $A:E\to F$ is $m$-\emph{completely continuous} provided that whenever $(e_n)_{n=1}^\infty\subset E$ is $m$-weakly null, $(Ae_n)_{n=1}^\infty$ is norm null.  We say a Banach space $Y$ has the $m$-\emph{Schur property} provided that $I_Y$ is $m$-completely continuous.  We collect the following facts. 
	
\begin{proposition}\begin{enumerate}[label=(\roman*)]\item If $X$ has TAP and $\mathfrak{A}_{0,\infty}$, then $X^*$ is $1$-Schur. \item If $E$ is a Banach space with $Sz(E)<\omega^\omega$ and $F$ is a Banach space with the $1$-Schur property, then $\mathfrak{L}(E,F)=\mathfrak{K}(E,F)$.  \end{enumerate}

\label{por}
\end{proposition}

\begin{proof}$(i)$ Let $\beta>0$ be such that $X$ has $\beta$-TAP.    Let $(x^*_n)_{n=1}^\infty$ be such that $\inf_n \|x^*_n\|>\ee \beta$.  Assume there exists a finite codimensional subspace $X_1$ of $X$ and a subsequence $(y^*_n)_{n=1}^\infty$ such that $\sup_n \|y^*_n|_{X_1}\| \leqslant \ee$.   By using TAP and replacing $X_1$ with a subspace, we may assume there exist projections $P^{X_1}_0, P^{X_1}_1:X\to X$ with norm  at most $\beta$ such that $P^{X_1}_0+P^{X_1}_1=I_X$ and $P^{X_1}_1(X)=X_1$.     Then for any $x\in B_X$, $P^{X_1}_1x \in \beta B_{X_1}$, so $|\langle y^*_n, x\rangle|\leqslant \ee \beta$. Therefore $\|(P^{X_1}_1)^*y^*_n\|\leqslant \ee\beta$.    Since $P^{X_1}_0$ is finite rank, \[\underset{n}{\lim\inf} \|y^*_n\| \leqslant \underset{n}{\lim} \|(P^{X_1}_0)^* y^*_n\| + \ee\beta =\ee\beta,\] a contradiction of the assumption that $\inf_n \|x^*_n\|>\ee \beta$.  Therefore for any finite codimensional subspace $X_1$ of $X$ and $n_0\in \nn$, there exist $x\in B_{X_1}$ and $n>n_0$ such that $\langle \text{Re\ }x^*_n, x\rangle>\ee$.     

Fix $\alpha>0$ and let $\psi_n$ be a winning strategy for Player I in the $A_\infty(\alpha,n)$ game.    Let $X_1=\psi_1(\varnothing)$.    Choose $n_1\in\nn$ and $x_1\in B_{X_1}$ such that $\text{Re\ }\langle x^*_{n_1}, x_1\rangle >\ee$.    Assuming $X_1, \ldots, X_{k-1}$ and $x_1, \ldots, x_{k-1}$,  have been defined, let \[X_k=\Bigl(\bigcap_{l=1}^k \psi_l(\varnothing)\Bigr)\cap\bigcap \Bigl\{\psi_l(x_{i_1}, \ldots, x_{i_m}): 1\leqslant m\leqslant l\leqslant k, 1\leqslant i_1<\ldots <i_m\leqslant k\Bigr\}\cap \Bigl(\bigcap_{i=1}^{k-1}\ker(x^*_{n_i})\Bigr).\] Choose $n_k>n_{k-1}$ and $x_k\in B_{X_k}$ such that $\text{Re\ }\langle x^*_{n_k}, x_k\rangle\geqslant \ee$.  This completes the recursive construction.    

Fix $l\leqslant k\in\nn$ and $k\leqslant i_1<\ldots < i_k$.    Note that $x_{i_1}\in B_{X_{i_1}}\subset B_X\cap \psi_k(\varnothing)$ and for $1\leqslant m<k$, $x_{i_{m+1}}\in B_{X_{i_{m+1}}} \subset B_X \cap \psi_k(x_{i_1}, \ldots, x_{i_m})$. Therefore $\|(x_{i_j})_{j=1}^l\|_1^w \leqslant \alpha$.    Therefore \begin{align*} \min\Bigl\{\Bigl\|\sum_{j=1}^k a_j x^*_{n_{i_j}}\Bigr\|: \sum_{j=1}^k |a_j|=1\Bigr\} & \geqslant \alpha^{-1}\min\Bigl\{\text{Re\ }\Bigl\langle \sum_{j=1}^k a_j x^*_{n_{i_j}}, \sum_{j=1}^k \text{sgn}(a_j) x_{i_j}\Bigr\rangle: \sum_{j=1}^k |a_j|=1\Bigr\} \\ & = \alpha^{-1}\Bigl\{\sum_{j=1}^k |a_j|\text{Re\ }\langle x^*_{n_{i_j}}, x_{i_j}\rangle: \sum_{j=1}^k |a_j|=1\Bigr\} \\ & \geqslant \ee/\alpha. \end{align*}  Since every non-empty member of $\mathcal{S}_1$ has the form $\{i_1, \ldots, i_k\}$ for some $k\leqslant i_1<\ldots <i_k$, we deduce that any weakly null sequence in $X^*$ which is not norm null has a subsequence which is an $\ell_1^1$-spreading model.  This concludes $(i)$. 

$(ii)$ It follows from \cite[Corollary $4.9$]{CN} that any Banach space which is $1$-Schur is $m$-Schur for every $m\in\nn$.  Assume $Sz(E)<\omega^\omega$ and $F$ is $1$-Schur.  Then for some $m\in\nn$, $Sz(E)<\omega^m$ and $F$ is $m$-Schur.   If $A:E\to Y$ fails to be compact, then there exists $\ee>0$ and $(e_n)_{n=1}^\infty\subset B_E$ such that $\inf_{n_1\neq n_2} \|e_{n_1}-e_{n_2}\|\geqslant \ee$.    Since $Sz(E)<\omega^\omega$, $E$ does not contain an isomorphic copy of $\ell_1$, we can assume without loss of generality that $(e_n)_{n=1}^\infty$ is weakly Cauchy, so $(f_n)_{n=1}^\infty := (e_{2n}-e_{2n-1})_{n=1}^\infty$ is bounded, weakly null, and not norm null.   Therefore some subsequence of $(f_n)_{n=1}^\infty$, which we may assume by relabeling is $(f_n)_{n=1}^\infty$ itself, is such that $(Af_n)_{n=1}^\infty$ is an $\ell_1^m+$-spreading model.  Therefore \begin{align*} 0 < \|A\|^{-1}\inf\{\|Af\|: G\in \mathcal{S}_m, f\in \text{co}\{f_n: n\in G\}\} \leqslant \inf\{\|f\|: \varnothing\neq G\in \mathcal{S}_m, f\in \text{co}\{f_n: n\in G\}\}.\end{align*}   Therefore if $y_G=f_{\max G}/2$ for $\varnothing\neq G\in \mathcal{S}_m$, then $(y_G)_{G\in \mathcal{S}_m\setminus\{\varnothing\}}\subset B_E$ is a weakly null tree and \[\inf \{\|y\|: \varnothing\neq G\in \mathcal{S}_m, y\in \text{co}\{y_H: H\preceq G\}\} = \inf\{\|f\|: \varnothing\neq G\in \mathcal{S}_m, f\in \text{co}\{f_n: n\in G\}\}>0.\]   Since $(y_G)_{G\in \mathcal{S}_m\setminus \{\varnothing\}}$ is weakly null and $\text{rank}(\mathcal{S}_m\setminus \{\varnothing\})=\omega^m$, $Sz(E)>\omega^m$ by Proposition \ref{fha}$(i)$. This contradiction finishes $(ii)$. 

\end{proof}

\begin{corollary} Fix $2\leqslant q<\infty$. For $N\in\nn$, if $X_1, \ldots, X_N$ are Banach spaces each of which has TAP, and $\mathfrak{A}_{0,\infty}$ (resp. $T_{0,\infty}$) and such that $X^*_i$ has AP and cotype $q$ for each $1\leqslant i\leqslant N$, then $(\tim_{\pi,i=1}^N X_i)^*=\tim_{\ee,i=1}^N X^*_i$.

\end{corollary}

\begin{proof} We prove by induction on $n$ that $(\tim_{\pi,i=1}^n X_i)^*=\tim_{\ee,i=1}^n X^*_i$.    For $n=1$, this is true by convention. Assume the result holds for some $n<N$. Then by Theorem \ref{upper1}, $Sz(\tim_{\pi,i=1}^n X_i)<\omega^\omega$.  By Proposition \ref{por}, $X^*_{n+1}$ is $1$-Schur, and therefore $\mathfrak{L}(\tim_{\pi,i=1}^n X_i, X^*_{n+1})=\mathfrak{K}(\tim_{\pi,i=1}^n, X_i, X^*_{n+1})$.  Since $X^*_{n+1}$ has AP,  \begin{align*} (\tim_{\pi,i=1}^{n+1} X_i)^* & = \mathfrak{L}(\tim_{\pi,i=1}^n X_i, X_{n+1}^*) = \mathfrak{K}(\tim_{\pi,i=1}^n X_i, X_{n+1}^*) = (\tim_{\pi,i=1}^n X_i)^*\tim_\ee X^*_{n+1} = \tim_{\ee,i=1}^{n+1} X^*_i. \end{align*} 

\end{proof}

\begin{corollary} If $X$ has $\mathfrak{A}_{0,\infty}$ and TAP and $X^*$ has AP and non-trivial cotype, then for all $N\in\nn$, $\mathcal{L}_N(X)=\tim_\ee^N X^*$ and $\mathcal{P}_N(X)=\tim_{\ee,s}^N X^*$. 
\label{sparse}
\end{corollary}

\section{The lower estimates}

In this section, we will use a general construction to show the sharpness of the estimates from the previous section. It follows from Proposition \ref{fha} that each of the following properties of a Banach space implies the next: \[\mathfrak{A}_{0,\infty} \Rightarrow \mathfrak{A}_{0,2} \Rightarrow \mathfrak{A}_{1,\infty} \Rightarrow \mathfrak{A}_{1,2}\Rightarrow \ldots.\]   In this section, we will show that for any $M,N\in\nn$ with $M>N$ and any infinite dimensional spaces $Y_1, \ldots, Y_M$, $\tim_{\pi,i=1}^M Y_i$ cannot have the $N^{th}$ property on the list above.  This is how we will show non-embeddability into $\tim_\pi^N c_0$.  We will also prove a non-embeddability result about $\tim_{\ee,i=1}^M Y_i$ into $\tim_\ee^N \ell_1$.  Before proceeding to the construction, we isolate the structures we will use to prove non-embeddability. First we recall some facts about the property $\mathfrak{A}_{N,p}$ and about the Szlenk index. 

\begin{proposition}\cite{C} Fix $1<p\leqslant \infty$ and let $1/p+1/q=1$. Fix $N\in \nn\cup \{0\}$.  \begin{enumerate}[label=(\roman*)]\item A Banach space $X$ has property $\mathfrak{A}_{N,p}$ if and only if there exists a constant $\alpha_0>0$ such that for any $n\in\nn$ and $\ee_1, \ldots, \ee_n>0$ such that \[s^{\omega^N}_{\ee_1}\ldots s^{\omega^N}_{\ee_n}(B_{X^*})\neq \varnothing,\] it follows that $\alpha_0^q \geqslant \sum_{i=1}^n \ee_i^q$. \item If $X$ has $\mathfrak{A}_{N,p}$, and if $Y$ is a subspace, quotient, or isomorph of $X$, then $Y$ also has $\mathfrak{A}_{N,p}$. \item If $Sz(X)\leqslant \omega^N$, and if $Y$ is a subspace, quotient, or isomorph of $X$, then $Sz(Y)\leqslant \omega^N$. \end{enumerate}

\label{facts}
\end{proposition}

\begin{proof}$(i)$ This is \cite[Theorem $4.11$]{C} applied with $\xi=N$.  In \cite[Theorem $4.11$]{C}, the property in $(i)$ is referred to as $N$-$q$-\emph{summable Szlenk index}, and it is defined for operators. The quantity $\alpha$ in the definition of property $\mathfrak{A}_{N,p}$ that we have given is the constant $\alpha_{N,p}(I_X)$ which appears in \cite[Theorem $4.11$]{C}. 

$(ii)$ In \cite[Theorem $4.11$]{C}, it was defined what it means for an operator to have $\mathfrak{A}_{N,p}$, and a Banach space has $\mathfrak{A}_{N,p}$ if and only if its identity has this property according to the operator definition. It was also shown there (Theorem $5.1$) that the class of operators which enjoy this property is an ideal, and that this ideal is injective and surjective (Proposition $5.5$).   If $X,Y$ are isomorphic and $I_X$ has $\mathfrak{A}_{N,p}$, then for some isomorphism $J:Y\to X$, $I_Y=J^{-1}I_XJ$ has property $\mathfrak{A}_{N,p}$ by the ideal property. It follows from injectivity of the ideal that any subspace of a space with $\mathfrak{A}_{N,p}$ also enjoys this property, and from surjectivity of the ideal that any quotient of a space with $\mathfrak{A}_{N,p}$ also enjoys this property. 

$(iii)$ These are all standard properties of Szlenk index. However, in \cite{B}, Brooker showed that for any ordinal $\xi$, the class of operators with Szlenk index not exceeding $\omega^\xi$ is an injective, surjective operator ideal. We did not define the Szlenk index of an operator, but we note here that the Szlenk index of an operator is defined in such a way that the Szlenk index of a Banach space is equal to the Szlenk index of its identity operator.  We deduce that any subspace, quotient, or isomorph of a space with Szlenk index not exceeding $\omega^N$ also enjoys this property using Brooker's result and an argument similar to that in the previous paragraph. 

\end{proof}

We next work to provide some characterizations of non-embeddability. We begin with the following easy fact. 

\begin{proposition} For any infinite-dimensional Banach space $Y_0$, any finite codimensional subspace $E$ of $Y_0$, and any weak$^*$-closed, finite codimensional subspace $F$ of $Y_0^*$, there exist $e\in B_E$ and $h\in 3B_F$ such that $\langle h,e\rangle=1$. 
\label{mhb}
\end{proposition}

\begin{proof} For each finite codimensional subspace $G$ of $E$, fix $e_G\in S_G$ and $f_G \in S_{X^*}$ such that $\langle f_G, e_G\rangle=1$.   If $D$ is the set of finite codimensional subspaces of $E$ directed by reverse inclusion, we may, after passing to a weak$^*$-convergent subnet, assume we have nets $(e_G)$, $(f_G)$ and $f\in B_{X^*}$ such that $(f_G-f)$ is weak$^*$-null.   Therefore we may select $h_G\in F$ such that $\|f_G-f-h_G\|\to 0$.   Since $\|f_G-f-h_G\|\to 0$ and $(e_G)$ is weakly null, there exists $G$ such that $\|f_G-f-h_G\|<1/7$ and $|\langle f, e_G\rangle|<1/7$.      Then \[|\langle h_G, e_G\rangle|\geqslant \langle f_G, e_G\rangle -|\langle f,e_G\rangle| - \|f_G-f-h_G\| \geqslant 5/7.\]   Then for an appropriate $\ee$ with $|\ee|=1$,  $e=e_G$, and $h= \frac{\ee h_G}{|\langle h_G, e\rangle|}$, \[\langle h, e\rangle = \frac{\ee\langle h_G, e\rangle}{|\langle h_G, e\rangle|} = \frac{|\langle h_G, e\rangle|}{|\langle h_G, e\rangle|}=1\] and \[\|h\| = \frac{\|h_G\|}{|\langle h_G, e\rangle|} \leqslant \frac{\|f_G\|+\|f\|+\|f_G-f-h_G\|}{2/7} \leqslant \frac{2+1/7}{5/7} = 3.\]

\end{proof}

The next result states that a certain type of local $c_0$ structure in a Banach space $Y$ implies the same order of asymptotic $c_0$ structure in $Y\tim_\ee Y_0$ whenever $Y_0$ is infinite dimensional.   

\begin{lemma}Let $Y$ be a Banach space and fix $N\in\nn$. Let $T$ be a well-founded tree and assume $(u_t)_{t\in T}\subset Y$, $(\upsilon_t)_{t\in T}\subset Y^*$ are such that $\langle \upsilon_t,u_t\rangle=1$ for all $t\in T$. \begin{enumerate}[label=(\roman*)]\item Suppose  $\|u_t\|\leqslant 1$ for all $t\in T$  and $\gamma>0$ is such that  $\|(\upsilon_s)_{s\preceq t}\|_1^w  \leqslant \gamma$ for all $t\in MAX(T)$. Then for any infinite dimensional Banach space $Y_0$ there exist a tree $S$ with $\text{rank}(S)=\text{rank}(T)$ and a weakly null collection $(v_s)_{s\in S}\subset B_{Y\tim_\pi Y_0}$ such that  \[\inf\Bigl\{\|v\|: r\in S, v\in \text{\emph{co}}\{v_r: r\preceq s\}\Bigr\}\geqslant 1/3\gamma.\]  \item If $\|\upsilon_t\|\leqslant 1$ for all $t\in T$ and $\gamma>0$ is such that $\| (u_s)_{s\preceq t}\|_1^w\leqslant \gamma$ for all $t\in MAX(T)$, then for any infinite dimensional Banach space $Y_0$, there exist a tree $S$ with $\text{rank}(S)=\text{rank}(T)$ and a weakly null collection $(v_s)_{s\in S} \subset B_{Y\tim_\ee Y_0}$ such that for every $s\in MAX(S)$, $\|(v_r)_{r\preceq s}\|_1^w\leqslant 1$ and, with $v_\varnothing=0$, $\|v_s-v_{s^-}\| \geqslant 1/3\gamma$ for all $s\in S$. \end{enumerate}

\label{nukemhi}
\end{lemma}

\begin{proof}$(i)$ Let $D$ be the set of finite codimensional subspaces of $Y_0$, directed by reverse inclusion.   Let \[S=\{(\nu_i, E_i)_{i=1}^n:(\nu_i)_{i=1}^n\in T, E_1, \ldots, E_n\in D\}.\]  It is quite clear that for any ordinal $\zeta$, \[S^\zeta=\{(\nu_i, E_i)_{i=1}^n: (\nu_i)_{i=1}^n\in T^\zeta, E_1, \ldots, E_n\in D\},\] which implies that $\text{rank}(S)=\text{rank}(T)$.      We define two collections $(y^*_s)_{s\in S}\subset B_{Y_0}$, $(y^*_s)_{s\in S}\subset 3B_{Y^*_0}$ such that $\langle y^*_s, y_s\rangle=1$ for all $s\in S$ and for any $r,s\in S$ with $r\prec s$, $\langle y^*_{r}, y_{s}\rangle =\langle y^*_{s}, y_{r}\rangle=0$.     We make these definitions by induction on $|s|$.    For $s\in S$ with $|s|=1$, write $s=(\nu, E)$. Choose $y_s\in B_E$ and $y^*_s\in B_{Y^*_0}\subset 3B_{Y^*_0}$ such that $\langle y^*_s, y_s\rangle=1$.  Now assume that for some $s\in S'$, $y_r, y^*_r$ have been defined for all $r\prec s$.    Write $s=(\nu_i, E_i)_{i=1}^n$ and let \[E=E_n \cap \bigcap_{r\prec s} \ker(y^*_r),\] which is a finite codimensional subspace of $Y_0$,  and \[F=\bigcap_{r\prec s}\ker(y_r),\] which is a weak$^*$-closed, finite codimensional subspace of $Y^*_0$.  Fix $y_s\in B_E$ and $y^*_s\in 3B_F$ such that $\langle y^*_s, y_s\rangle$.    This completes the recursive construction. It is clear from the construction that $\langle y^*_r,y_s\rangle=\langle y^*_s,y_r\rangle=0$ for any $r,s\in S$ with $r\prec s$.

For $s=(\nu_i, E_i)_{i=1}^n\in S$, let $t_s=(\nu_i)_{i=1}^n\in T$. Note that by the properties of $T$ discussed at the beginning of the previous paragraph, $t_s\in MAX(T)$ if and only if $s\in MAX(S)$.   Note also that $(u_{t_s}\otimes y_s)_{s\in S}\subset B_{Y\tim_\pi Y_0}$, since $\|u_t\|\leqslant 1$ for all $t\in T$ and $\|y_s\|\leqslant 1$ for all $s\in S$.   We next claim that $(u_{t_s}\otimes y_s)_{s\in S}$ is a weakly null collection. To see that, suppose that for some ordinal $\zeta$, $s\in (\{\varnothing\}\cup S)^{\zeta+1}$. This implies that $t_s\in (\{\varnothing\}\cup T)^{\zeta+1}$, which means there exists $\nu$ such that the concatenation $t_s\smallfrown(\nu)$ is a member of $T^\zeta$. Then for any $E\in D$, $s\smallfrown (\nu, E)\in T^\zeta$. Since $y_{s\smallfrown (\nu, E)}\in E$, $(y_{s\smallfrown (\nu,E)})_{E\in D}$ is a weakly null net.    Since $t_{s\smallfrown(\nu,E)}=t_s\smallfrown(\nu)$ for all $E\in D$, $(u_{t_{s\smallfrown(\nu,E)}}\otimes y_{s\smallfrown(\nu,E)})_{E\in D}=(u_{t_s\smallfrown (\nu)}\otimes y_{s\smallfrown(\nu,E)})_{E\in D}$ is a weakly null net. Since $(s\smallfrown (\nu,E))^-=s$ for all $E\in D$, \[0\in \overline{\{u_{t_{s\smallfrown(\nu,E)}}\otimes y_{s\smallfrown(\nu,E)}: E\in D\}}^\text{weak} \subset \overline{\{u_{t_r}\otimes y_r: r\in S^\zeta, r^-=s\}}^\text{weak},\] and $(u_{t_s}\otimes y_s)_{s\in S}$ is weakly null, as claimed.

By Proposition \ref{pupper}$(ii)$, for any $s\in MAX(S)$, \[\|(\upsilon_{t_r}\otimes y^*_r)_{r\preceq s}\|_{\ee,1}^w \leqslant \|(\upsilon_{t_r})_{r\preceq s}\|_1^w \|(y^*_r)_{r\preceq s}\|_\infty = 3\|(\upsilon_r)_{r\preceq t_s}\|_1^w \leqslant 3\gamma.\]  Therefore $\|\sum_{r\preceq s} \upsilon_{t_r}\otimes y^*_s\|_\ee\leqslant 3\gamma$.    Next note that for any $s\in MAX(S)$ and $r, r'\preceq s$ such that $r\neq r'$, \[\langle \upsilon_{t_r}\otimes y^*_r, u_{t_r}\otimes y_r\rangle = \langle \upsilon_{t_r}, u_{t_r}\rangle\langle y^*_r, y_r\rangle = 1\cdot 1=1\]  and \[\langle \upsilon_{t_{r'}}\otimes y^*_{r'}, u_{t_r}\otimes y_r\rangle = \langle \upsilon_{t_{r'}}, u_{t_r}\rangle\langle y^*_{r'}, y_r\rangle = \langle \upsilon_{t_{r'}}, u_{t_r}\rangle\cdot 0=0.\]  Therefore for any $r\preceq s\in MAX(S)$, \[ \Bigl\langle \sum_{r'\preceq s} \upsilon_{t_{r'}}\otimes y^*_{r'}, u_{t_r}\otimes y_r\Bigr\rangle =1.\]   From this it follows that for any $y\in \text{co}\{u_{t_r}\otimes y_r: r\preceq s\}$, $\|y\|_\pi \geqslant 1/3\gamma$.   

$(ii)$ We define $S$, $(y_s)_{s\in S}\subset B_{Y_0}$, and $(y^*_s)_{s\in S}\subset 3B_{Y^*_0}$ exactly as in $(i)$.   Let $v_\varnothing=0$ and let $v_s= u_{t_s}\otimes y_s/\gamma\in Y\tim_\ee Y_0$. Note that for any $s\in S$, $\langle \upsilon_{t_s}\otimes y^*_s, v_{s^-}\rangle=0$. This follows from properties of the construction of $(i)$ if $s^-\in S$, and follows from the fact that $v_\varnothing=0$ if $s^-=\varnothing$.  As in $(i)$, for any $s\in MAX(S)$, \[\|(v_r)_{r\preceq s}\|_1^w \leqslant \|(u_r)_{r\preceq t_s}\|_1^w \|(y_r)_{r\preceq s}\|_\infty /\gamma.\]  Since $\|\upsilon_{t_s}\otimes y^*_s\|=\|\upsilon_{t_s}\|\|y^*_s\| \leqslant 1\cdot 3=3$, it follows that for any $s\in S$, \[\|v_s-v_{s^-}\| \geqslant \langle \upsilon_{t_s}\otimes y_s^*/3, v_s-v_{s^-}\rangle =\langle \upsilon_{t_s}\otimes y_s^*, u_{t_s}\otimes y_s\rangle/3\gamma=1/3\gamma.\]

\end{proof}

\begin{lemma} Fix $N\in\nn$ and $1\leqslant p,q\leqslant \infty$ with $1/p+1/q=1$.   \begin{enumerate}[label=(\roman*)]\item Let $Z_1, \ldots, Z_N$ be Banach spaces such that $Z_1^*, \ldots, Z^*_{N-1}$ have trivial cotype and $\ell_p$ is finitely representable in $Z^*_N$.  Then for  any $\gamma>1$ and $n\in\nn$, there exist a tree $T$ with $\text{rank}(T)=\omega^{N-1}n$,  $(z_t)_{t\in T}\subset \tim_{\pi,i=1}^{N}Z_i$, $(\zeta_t)_{t\in T} \subset \tim_{\ee,i=1}^{N} Z_i^* \subset (\tim_{\pi,i=1}^{N} Z_i)^*$, and $(z^{**}_t)_{t\in T} \subset \tim_{\pi,=1}^{N}  Z_i^{**} \subset (\tim_{\ee,i=1}^{N} Z_i^*)^*$ such that for all $t\in T$, $\|z_t\|_\pi, \|z_t^{**}\|_\pi \leqslant 1$,  $\langle \zeta_t, z_t\rangle=\langle z_t^{**}, \zeta_t\rangle=1$, and  $\|(\zeta_s)_{s\preceq t}\|_1^w\leqslant \gamma n^{1/p}$.  If $p=\infty$, the tree $T$ can be taken to have rank $\omega^N$. 

\item Let $Z_1, \ldots, Z_N$ be Banach spaces such that $Z_1, \ldots, Z_{N-1}$ have trivial cotype and $\ell_p$ is finitely representable in $Z_N$.  Then for  any $\gamma>1$ and $n\in\nn$, there exist a tree $T$ with $\text{rank}(T)=\omega^{N-1}n$,  $(z_t)_{t\in T}\subset \tim_{\ee,i=1}^{N}Z_i$, and  $(\zeta_t)_{t\in T} \subset \tim_{\pi,i=1}^{N} Z_i^* \subset (\tim_{\pi,i=1}^{N} Z_i)^*$ such that for all $t\in T$, $\|\zeta_t\|_\pi\leqslant 1$,  $\langle \zeta_t, z_t\rangle=1$, and  $\|(z_s)_{s\preceq t}\|_1^w\leqslant \gamma n^{1/p}$.  If $p=\infty$, the tree $T$ can be taken to have rank $\omega^N$.

\end{enumerate}

Moreover, all tensors above can be taken to be elementary tensors. 

\label{subhmnd}
\end{lemma}

\begin{proof} We note first that all tensors defined in the proof will be elementary, so that the ``moreover'' statement will be satisfied by our construction. 

$(i)$ We work by induction. Assume $N=1$, in which case we omit reference to $Z_1, \ldots, Z_{N-1}$ and interpret all tensor products as $Z_N=Z_1$.    Fix $n\in\nn$ and let \[T_n=\{(n-1, \ldots, n-m): 1\leqslant m\leqslant n\},\] which is a tree of rank $n$.    Fix $1<\gamma_0<\gamma$.  By hypothesis, we can fix $(\zeta_{m,n})_{m=1}^n\subset Z_1^*$ such that for all scalars $(a_m)_{m=1}^n$, \[\gamma_0 \|(a_m)_{m=1}^n\|_p \leqslant \Bigl\|\sum_{m=1}^n a_m\zeta_{m,n}\Bigr\|\leqslant \gamma\|(a_m)_{m=1}^n\|_p.\]  Note that $\text{rank}(T_n)=n$.   We can fix $(z^{**}_{m,n})_{m=1}^n\subset Z_1^{**}$ to be Hahn-Banach extensions of the biorthogonal functionals to $(\zeta_{m,n})_{m=1}^n$. We note that $\|z_{m,n}\|\leqslant 1/\gamma_0<1$ for all $1\leqslant m\leqslant n$.     By local reflexivity, we can select $(z_{m,n})_{m=1}^n\subset B_{Z_1}$ such that $\langle \zeta_{m,n}, z_{m,n}\rangle=\langle z^{**}_{m,n}, \zeta_{m,n}\rangle=1$ for all $1\leqslant m\leqslant n$.     Then for any $t=(n-1, \ldots, n-m)\in T_n$, let $z_t=z_{m,n}$, $\zeta_t=\zeta_{m,n}$, $z^{**}_t=z^{**}_{m,n}$. Then $\|z_t\|_\pi, \|z^{**}_t\|_\pi\leqslant 1$, $\langle z^{**}_t, \zeta_t\rangle=\langle \zeta_t, z_t\rangle=1$, and  \[\|(\zeta_s)_{s\preceq t}\|_1^w=\|(\zeta_{l,n})_{l=1}^m\|_1^w \leqslant n^{1/p}\|(\zeta_{l,n})_{l=1}^m\|_q^w \leqslant \gamma n^{1/p}.\]  If $p<\infty$, this finishes the base case. If $p=\infty$, then we have $\|(\zeta_s)_{s\preceq t}\|_1^w \leqslant \gamma$, which is independent of $n$.  Our tree $T$ of rank $\omega=\omega^1=\omega^N$ is $T=\cup_{n=1}^\infty T_n$.

Assume the result holds for $N$ and assume $Z_1, \ldots, Z_N$ are such that $Z_1^*, \ldots, Z_N^*$ have trivial cotype and $\ell_p$ is finitely representable in $Z^*_{N+1}$.     Since $\ell_\infty$ is finitely representable in $Z_N$, then the last sentence of $(i)$ together with the inductive hypothesis yield the existence of a tree $R$ with rank $\omega^N$ and collections $(y_t)_{t\in R}\subset B_{\tim_{\pi,i=1}^N Z_i}$, $(\psi_t)_{t\in R} \subset \tim_{\ee,i=1}^N Z_i^*$, and $(y^{**}_t)_{t\in R}\subset B_{\tim_{\pi,i=1}^N Z_i^{**}}$ such that $\langle y^{**}_t, \psi_t\rangle = \langle \psi_t, y_t\rangle=1$ and $\|(\psi_s)_{s\preceq t}\|_1^w \leqslant \gamma^{1/2}$ for all $t\in R$.   Fix $n\in\nn$ and  let $T_n$ denote the set of all sequences of pairs $t=(r_j, \nu_j)_{j=1}^k$ such that there exist $1\leqslant m\leqslant n$ and  an interval partition $I_{n-1}, \ldots, I_{n-m}$ of $\{1, \ldots, k\}$ such that \begin{enumerate}[label=(\alph*)]\item $I_{n-1}<\ldots <I_{n-m}$, \item for each $1\leqslant l\leqslant m$, $r_j=n-l$ for each $j\in I_{n-l}$, \item for each $1\leqslant l\leqslant m$, $(\nu_j)_{j\in I_{n-l}}\in R$. \end{enumerate}  Condition (a) implies that $(r_j)_{j=1}^k$ is a non-increasing sequence of integers.  Condition (b) implies that $m$ and the interval partition $I_{n-1}, \ldots, I_{n-m}$ are uniquely determined.  For $t\in T_n$, let us call this $m$ the \emph{outer length} of the sequence $t$.    It is easy to see that for each $1\leqslant l \leqslant n$, $T_n^{\omega^N (l-1)}$ consists of those sequences in $T_n$ whose outer length does not exceed $n-l$ (equivalently, those sequences $(r_j, \nu_j)_{j=1}^k$ such that $r_k\geqslant l$). From this it follows that $\text{rank}(T_n)=\omega^Nn$.  For each $n\in\nn$, since $\ell_p$ is finitely representable in $Z_{N+1}^*$, we can fix $1<\gamma_0<\gamma^{1/2}$ and a sequence $(\zeta_{m,n})_{m=1}^n\subset Z_{N+1}^*$ such that for all scalars $(a_m)_{m=1}^n$, \[\gamma_0\|(a_m)_{m=1}^n\|_p \leqslant \Bigl\|\sum_{m=1}^n a_m \zeta_{m,n}\Bigr\|\leqslant \gamma^{1/2} \|(a_m)_{m=1}^n\|_p.\]    We can then choose $(z^{**}_{m,n})_{m=1}^n\subset B_{Z_{N+1}^{**}}$ and $(z_{m,n})_{m=1}^n\subset B_{Z_{N+1}}$ to each be biorthogonal to $(\zeta_{m,n})_{m=1}^n$ exactly as in the previous paragraph. 

 For $t=(r_j, \nu_j)_{j=1}^k\in T_n$ with outer length $m$ and interval partition $I_{n-1}, \ldots, I_{n-m}$, define \[z_t=y_{(\nu_j)_{j\in I_{n-m}}}\otimes z_{m,n},\] \[\zeta_t=\psi_{(\nu_j)_{j\in I_{n-m}}}\otimes \zeta_{m,n} ,\] and \[z^{**}_t = y^{**}_{(\nu_j)_{j\in I_{n-m}}}\otimes z^{**}_{m,n}.\]  Of course, $\|z_t\|_\pi, \|z^{**}_t\|_\pi \leqslant 1$ and $\langle z^{**}_t, \zeta_t\rangle=\langle \zeta_t, z_t\rangle=1$ for all $t\in T$.

It remains to show the upper estimate on the weakly $1$-summing norm. Fix $t=(r_j, \nu_j)_{j=1}^k\in T_n$. Let $m$ be the outer length of $t$ and let $I_{n-1}, \ldots, I_{n-m}$ be the interval partition of $\{1, \ldots, k\}$ associated with $t$. Choose $0=k_0<\ldots <k_m$ such that $I_{n-l}=(k_{l-1}, k_l]$ for each $1\leqslant l\leqslant m$.   For each $1\leqslant i\leqslant k$, let $t|i=(r_j, \nu_j)_{j=1}^i$, so $\{r\in T_n: r\preceq t\}=\{t|i: 1\leqslant i\leqslant k\}$.  For each $1\leqslant l\leqslant m$, let $t_l=(\nu_j)_{j\in I_{n-l}}\in R$. For $1\leqslant l\leqslant m$ and  $h\in [1, k_l-k_{l-1}]$ (equivalently, $1\leqslant h\leqslant |I_{n-l}|$), let $I_{n-l}|h$ be the initial segment of $I_{n-l}$ of cardinality $h$.        Note that for $1\leqslant l\leqslant m$ and $h\in [1, k_l-k_{l-1}]$, $k_{l-1}+h\in I_{j-l}=(k_{l-1}, k_l]$, $t|k_{l-1}+h=(r_j, \nu_j)_{j=1}^{k_{l-1}+h}$ has outer length $l$ and interval partition $I_{n-1}, \ldots, I_{n-l+1}, I_{n-l}|h$.  Therefore for such $l$ and $h$, \[\zeta_{t|k_{l-1}+h}= \psi_{(\nu_j)_{j\in I_{n-l}|h}} \otimes \zeta_{l,n}=\psi_{(\nu_j)_{t_l|h}} \otimes \zeta_{l,n}.\]  By our choice of $R$ and $(\zeta_r)_{r\in R}$, it holds that for each $1\leqslant l\leqslant m$,  $\|(\psi_{t_l|h})_{h=1}^{k_l-k_{l-1}}\|_1^w \leqslant \gamma^{1/2}$. From this it follows that for any scalars $(\ee_h)_{h=1}^{k_l-k_{l-1}}$ such that $|\ee_h|=1$ for all $1\leqslant h\leqslant k_l-k_{l-1}$, $\|\sum_{h=1}^{k_l-k_{l-1}} \ee_h \psi_{t_l|h}\|_\ee \leqslant \gamma^{1/2}$.   Using this bit of bookkeeping together with Proposition \ref{pupper}$(iii)$, we deduce that for any scalars $(\ee_j)_{j=1}^k$ such that $|\ee_j|=1$ for each $1\leqslant j\leqslant k$,

\begin{align*} \Bigl\|\sum_{r\preceq t} \ee_{|r|} \zeta_r \Bigr\|_\ee & = \Bigl\|\sum_{j=1}^k \ee_j \zeta_{t|j}\Bigr\|_\ee = \Bigl\|\sum_{l=1}^m \sum_{h=1}^{k_l-k_{l-1}} \ee_{k_{l-1}+h} \zeta_{t|k_{l-1}+h}\Bigr\|_\ee \\ & = \Bigl\|\sum_{l=1}^m \sum_{h=1}^{k_l-k_{l-1}} \ee_{k_{l-1}+h} \psi_{t_l|h}\otimes \zeta_{l,n}\Bigr\|_\ee \\ & = \Bigl\|\sum_{l=1}^m \Bigl(\sum_{h=1}^{k_l-k_{l-1}} \ee_{k_{l-1}+h} \psi_{t_l|h}\Bigr)\otimes \zeta_{l,n}\Bigr\|_\ee \\ & \leqslant \Bigl\|\Bigl(\sum_{h=1}^{k_l-k_{l-1}} \ee_{k_{l-1}+h} \psi_{t_l|h}\Bigr)_{l=1}^m\Bigr\|_{\ee,\infty} \|(\zeta_{l,n})_{l=1}^m\|_1^w \\ & \leqslant \gamma_{1/2} \cdot n^{1/p}\|(\zeta_{l,n})_{l=1}^m\|_q^w \leqslant \gamma n^{1/p}. \end{align*}  If $p<\infty$, this finishes $(i)$.    As in the case $N=1$, if $p=\infty$, we let $T=\cup_{n=1}^\infty T_n$.

$(ii)$ Under the hypotheses of $(ii)$, we can prove the existence of $(z_t)_{t\in T_n}$ and $(\zeta_t)_{t\in T_n}$ using the same methods we used to prove the existence of $(\zeta_t)_{t\in T_n}$ and $(z^{**}_t)_{t\in T_n}$ in $(i)$.

\end{proof}

We next show that collections such as those in Lemma \ref{nukemhi} can be found in higher tensor products, due to Dvoretzky's theorem. 

\begin{corollary} Fix $N\in\nn$.  \begin{enumerate}[label=(\roman*)]\item For any infinite dimensional Banach spaces $Y_1, \ldots, Y_{2N}$ and any $\gamma>1$, there exist a tree $T$ with $\text{rank}(T)=\omega^N$,  $(u_t)_{t\in T}\subset \tim_{\pi,i=1}^{2N}Y_i$, $(\upsilon_t)_{t\in T} \subset \tim_{\ee,i=1}^{2N} Y_i^* \subset (\tim_{\pi,i=1}^{2N} Y_i)^*$, and $(u^{**}_t)_{t\in T} \subset \tim_{\pi,=1}^{2N}  Y_i^{**} \subset (\tim_{\ee,i=1}^{2N} Y_i^*)^*$ such that for all $t\in T$, $\|u_t\|_\pi, \|u_t^{**}\|_\pi \leqslant 1$ and $\langle \upsilon_t, u_t\rangle=\langle u_t^{**}, \upsilon_t\rangle=1$, and for each $t\in T$, $\|(\upsilon_s)_{s\preceq t}\|_1^w\leqslant \gamma$. 

\item For any infinite dimensional Banach spaces $Y_1, \ldots, Y_{2N}$ and any $\gamma>1$, there exist a tree $T$ with $\text{rank}(T)=\omega^N$,  $(u_t)_{t\in T}\subset \tim_{\ee,i=1}^{2N}Y_i$, and $(\upsilon_t)_{t\in T} \subset \tim_{\pi,i=1}^{2N} Y_i^* \subset (\tim_{\ee,i=1}^{2N} Y_i)^*$ such that for all $t\in T$, $\|\upsilon_t\|_\pi\leqslant 1$ and $\langle \upsilon_t, u_t\rangle=1$, and for each $t\in T$, $\|(u_s)_{s\preceq t}\|_1^w\leqslant \gamma$.

\item For any infinite dimensional Banach spaces $Y_1, \ldots, Y_{2N-1}$, any $\gamma>1$, and any $j\in\nn$, there exist a tree $T_j$ with $\text{rank}(T_j)=\omega^{N-1}j$,  $(u_t)_{t\in T_j}\subset \tim_{\pi,i=1}^{2N-1}Y_i$, $(\upsilon_t)_{t\in T_j} \subset \tim_{\ee,i=1}^{2N-1} Y_i^* \subset (\tim_{\pi,i=1}^{2N-1} Y_i)^*$, and $(u^{**}_t)_{t\in T_j} \subset \tim_{\pi,=1}^{2N-1}  Y_i^{**} \subset (\tim_{\ee,i=1}^{2N-1} Y_i^*)^*$ such that for all $j\in\nn$ and $t\in T_j$, $\|u_t\|_\pi, \|u_t^{**}\|_\pi \leqslant 1$ and $\langle \upsilon_t, u_t\rangle=\langle u_t^{**}, \upsilon_t\rangle=1$, and $\|(\upsilon_s)_{s\preceq t}\|_1^w\leqslant \gamma j^{1/2}$. 

\item For any infinite dimensional Banach spaces $Y_1, \ldots, Y_{2N-1}$,  any $\gamma>1$, and any $j\in\nn$,  there exist a tree $T_j$ with $\text{rank}(T_j)=\omega^{N-1}j$,  $(u_t)_{t\in T_j}\subset \tim_{\ee,i=1}^{2N-1}Y_i$, and $(\upsilon_t)_{t\in T_j} \subset \tim_{\pi,i=1}^{2N-1} Y_i^* \subset (\tim_{\ee,i=1}^{2N-1} Y_i)^*$ such that for all $t\in T_j$, $\|\upsilon_t\|_\pi\leqslant 1$ and $\langle \upsilon_t, u_t\rangle=1$, and for each $t\in T_j$, $\|(u_s)_{s\preceq t}\|_1^w\leqslant \gamma j^{1/2}$.

\end{enumerate}

Moreover, all tensors above can be taken to be elementary tensors. 

\label{brnsky}
\end{corollary}

\begin{proof} By Dvoretzky's Theorem \cite{Dv},  $\ell_2$ is finitely representable in any infinite dimensional Banach space.  Moreover, for any infinite dimensional Banach spaces $E,F$, $n\in\nn$, and  $\gamma>1$, we can find sequences $(e_i)_{i=1}^n \subset E$, $(e^*_i)_{i=1}^n \subset B_{E^*}$, $(f_i)_{i=1}^n\subset B_F$, $(f^*_i)_{i=1}^n \subset B_{F^*}$ such that $\|(e_i)_{i=1}^n\|_2^w, \|(f_i)_{i=1}^n\|_2^w \leqslant \gamma^{1/2}$ and $\langle e^*_i, e_j\rangle = \langle f^*_i, f_j\rangle \delta_{ij}$ for all $1\leqslant i,j\leqslant n$.   By Proposition \ref{pupper}$(i)$, $\|(e_i\otimes f_i)_{i=1}^n\|_{\ee,1}^w \leqslant \gamma$. Since \[\Bigl\|\sum_{i=1}^n a_i e_i\otimes f_i\Bigr\|_\ee \geqslant \max_{1\leqslant j \leqslant n} \Bigl|\Bigl\langle e_j^*\otimes f^*_j, \sum_{i=1}^na_ie_i\otimes f_i\Bigr\rangle \Bigr| = \max_{1\leqslant j\leqslant n}|a_j|,\] it follows that $E\tim_\ee F$ has trivial cotype.  Therefore for any infinite dimensional Banach spaces $Y_1, Y_2, \ldots$, we deduce Corollary \ref{brnsky} from Lemma \ref{subhmnd} by noting that for each $i\in\nn$, $Y_{2i-1}\tim_\ee Y_{2i}$ each have trivial cotype and $\ell_2$ is finitely representable in $Y_i$.

\end{proof}

\begin{theorem} Let $Y_0, Y_1, \ldots$ be infinite dimensional Banach spaces. Let $X$ be a Banach space and fix $N\in\nn$. \begin{enumerate}[label=(\roman*)]\item If $Sz(X)\leqslant \omega^N$,  $\tim_{\pi,i=0}^{2N} Y_i$ is not isomorphic to any subspace of any quotient of $X$.  \item If  $Sz(X)\leqslant \omega^N$, then $\tim_{\ee,i=0}^{2N} Y_i$ is not isomorphic to any subspace of $X^*$. \item  If $X$ has property $\mathfrak{A}_{N-1,\infty}$, then $\tim_{\pi,i=0}^{2N-1} Y_i$ is not isomorphic to any subspace of any quotient of $X$. \item If $X$ has property $\mathfrak{A}_{N-1,\infty}$, then $\tim_{\ee,i=0}^{2N-1} Y_i$ is not isomorphic to any subspace of $X^*$.  \end{enumerate}
\label{thm}
\end{theorem}

\begin{proof}$(i)$ Combining Corollary \ref{brnsky}$(i)$ with $\gamma=2$ and Lemma \ref{nukemhi}$(i)$ applied with $Y=\tim_{\pi,i=1}^{2N}Y_i$,  there exist a tree $S$ with $\text{rank}(S)=\omega^N$ and a weakly null collection $(v_s)_{s\in S}\subset B_{Y\tim_\pi Y_0}$ such that \[\inf \{\|y\|: s\in MAX(S), y\in \text{co}\{v_r: r\preceq s\}\} \geqslant 1/6.\] By Proposition \ref{fha}$(i)$, $Sz(\tim_{\pi,i=0}^{2N} Y_i)>\omega^N$. By Proposition \ref{facts}$(iii)$, $\tim_{\pi,i=0}^{2N} Y_i$ is not isomorphic to any subspace of any quotient of a Banach space whose Szlenk index does not exceed $\omega^N$.

$(ii)$ Assume $A:\tim_{\ee,i=0}^{2N} Y_i \to X^*$ is an isomorphic embedding. Without loss of generality, we may assume that for some $c>0$, $c\|u\| \leqslant \|Au\|\leqslant \|u\|$ for all $u\in \tim_{\ee,i=0}^{2N} Y_i$.     Combining Corollary \ref{brnsky}$(ii)$ with $\gamma\in(1,2)$ and Lemma \ref{nukemhi}$(ii)$ applied with $Y=\tim_{\ee,i=1}^{2N} Y_i$, there exist a tree $S$ with $\text{rank}(S)=\omega^N$ and a weakly null collection $(v_s)_{s\in S}$ such that, with $v_\varnothing=0$, $\|\sum_{r\preceq s} v_r\|\leqslant 1$ and $\|v_s\|> 1/6$ for all $s\in S$.   Note that $(Av_s)_{s\in S}$ is weakly null, and therefore weak$^*$-null, in $X^*$.  Moreover, $\sum_{r\preceq s}Av_r\in B_{X^*}$ and $\|Av_s-v_{s^-}\|>c/6$ for all $s\in S$.    From this and an easy induction argument, it follows that for each ordinal $\zeta$ and each $s\in (\{\varnothing\}\cup S)^\zeta$, $Av_\varnothing+\sum_{r\preceq s} Av_r\in s^\zeta_{c/6}(B_{X^*})$. In particular, since $\varnothing\in (\{\varnothing\}\cup S)^{\omega^N}$, $0=Av_\varnothing\in s^{\omega^N}_{c/6}(B_{X^*})$, and $Sz(X)>\omega^N$.   Let us give the details of the induction.   The $\zeta=0$ case holds because for any $s\in (\{\varnothing\}\cup S)$, \[Av_\varnothing+\sum_{r\preceq s}Av_r =0\in B_{X^*}=s^0_{c/6}(B_{X^*})\] if $s=\varnothing$, and otherwise \[\Bigl\|Av_\varnothing+\sum_{r\preceq s}Av_r \Bigr\| \leqslant \Bigl\|v_\varnothing+\sum_{r\preceq s} v_r\Bigr\|= \Bigl\|\sum_{r\preceq s}v_r\Bigr\|\leqslant 1, \] so $Av_\varnothing+\sum_{r\preceq s}Av_r\in s^0_{c/6}(B_{X^*})$.  If $\zeta$ is a limit ordinal and $s\in (\{\varnothing\}\cup S)^\zeta=\cap_{\nu<\zeta}(\{\varnothing\}\cup S)^\nu$, then by the inductive hypothesis, \[Av_\varnothing+\sum_{r\preceq s}Av_r \in \bigcap_{\nu<\zeta} s^\nu_{c/6}(B_{X^*})=s_{c/6}^\zeta(B_{X^*}).\]     If the result holds for $\zeta$ and $s\in (\{\varnothing\}\cup S)^{\zeta+1}$, then by the definition of weakly null trees, \[0\in \overline{\{v_t:t\in S^\zeta, t^-=s\}}^\text{weak}.\] By weak-weak$^*$ continuity of $A$, \[0\in \overline{\{Av_t:t\in S^\zeta, t^-=s\}}^{\text{weak}^*}.\]Let $V$ be any weak$^*$-neighborhood of $Av_\varnothing+\sum_{r\preceq s}Av_r$ in $X^*$.    Then $V-Av_\varnothing+\sum_{r\preceq s}Av_r$ is a weak$^*$-neighborhood of $0$ in $X^*$. Since \[0\in \overline{\{Av_t:t\in S^\zeta, t^-=s\}}^{\text{weak}^*},\] there exists $t\in S^\zeta$ such that $t^-=s$ and $Av_t\in V-Av_\varnothing+\sum_{r\preceq s}Av_r$, so \[Av_t+Av_\varnothing+\sum_{r\preceq s}Av_r\in V.\]   Since $t^-=v$, \[Av_t+Av_\varnothing + \sum_{r\preceq s}Av_r = Av_\varnothing+\sum_{r\preceq t} Av_r.\]  Since $t\in S^\zeta$, the inductive hypothesis yields that $Av_\varnothing+\sum_{r\preceq t}Av_r\in s^\zeta_{c/6}(B_{X^*})$.    Therefore $Av_\varnothing+\sum_{r\preceq t}Av_r\in V\cap s^\zeta_{c/6}(B_{X^*})$.    Since \[Av_\varnothing+\sum_{r\preceq t}Av_r-Av_\varnothing+\sum_{r\preceq s}Av_r=Av_t\] and \[\|Av_t\|\geqslant c\|v_t\|>c/6,\] $\text{diam}(B\cap s^\zeta_{c/6}(B_{X^*}))> c/6$.  Since $V$ was an arbitrary weak$^*$-neighborhood of $Av_\varnothing+\sum_{r\preceq s}Av_r$, $Av_\varnothing+\sum_{r\preceq s}Av_r\in s^{\zeta+1}_{c/6}(B_{X^*})$.

$(iii)$ Fix $1\leqslant p,q\leqslant \infty$ such that $1/p+1/q=1$ and  $2<p\leqslant \infty$. Assume that $\tim_{\pi,i=0}^{2N-1}Y_i$ has $\mathfrak{A}_{N-1,p}$ and let $\alpha>0$ be such that for any $j\in\nn$,  any  tree $T$ with $\text{rank}(T)=\omega^{N-1}j$, and any weakly null collection $(y_t)_{t\in T}\subset B_{\tim_{\pi,i=0}^{2N-1}Y_i}$, there exist $\varnothing=t_0\prec \ldots \prec t_n$ such that $t_i\in MAX(T^{\omega^{N-1}(n-i)})$ and $y_i\in\text{co}\{y_s: t_{i-1}\prec s\preceq t_i\}$ such that $\|(y_i)_{i=1}^n \|_q^w \leqslant \alpha$. The latter inequality implies that for this $(y_i)_{i=1}^n$, the map $T:\ell_p^n\to \tim_{\pi,i=0}^{2N-1}Y_i$ given by $Te_i=y_i$ satisfies $\|T\|\leqslant \alpha$. In particular, \[\|\sum_{i=1}^n y_i\|\leqslant \alpha n^{1/p}.\]  

Since $q<2$, it follows that  $\frac{1}{q}-\frac{1}{2}>0$. Choose   $j\in\nn$ so large that $j^{\frac{1}{q}-\frac{1}{2}}> 6\alpha$.   Combining Corollary \ref{brnsky}$(iii)$ with $\gamma=2$ and Lemma \ref{nukemhi}$(i)$ with $Y=\tim_{\pi, i=1}^{2N-1} Y_i$, there exist a tree $S$ with $\text{rank}(S)=\omega^{N-1}j$ and a weakly null collection $(y_s)_{s\in S}\subset B_{\tim_{\pi,i=0}^{2N-1} Y_i}$ such that \[\inf\{\|y\|: s\in S, y\in\text{co}\{y_r: r\preceq s\}\} \geqslant 1/6j^{1/2}.\]    We can fix $\varnothing=s_0\prec \ldots \prec s_j$ and $y_i\in \text{co}\{y_r: s_{i-1}\prec r\preceq s_i\}$ such that $\|(y_i)_{i=1}^j\|_q^w\leqslant \alpha$.  Since $y=\frac{1}{j}\sum_{i=1}^j y_j\in \text{co}\{y_r: r\preceq s_j\}$, \[ j^{-1/2}/6  \leqslant  \|y\| \leqslant \frac{1}{j}\cdot \alpha j^{1/q} = \alpha j^{-1/q}.\]     This contradicts our choice of $j$, and this contradiction yields that $\tim_{\pi,i=0}^{2N-1}Y_i$ does not have $\mathfrak{A}_{N-1,p}$. Applying this with $p=\infty$, $\tim_{\pi,i=0}^{2N-1}Y_i$ does not have property $\mathfrak{A}_{N-1,\infty}$.       By Proposition \ref{facts}$(ii)$, if $X$ has $\mathfrak{A}_{N-1,\infty}$, $\tim_{\pi,i=0}^{2N-1}Y_i$ is not isomorphic to any subspace of any quotient of $X$.

$(iv)$ Fix $1\leqslant p,q\leqslant \infty$ such that $1/p+1/q=1$ and $2<p\leqslant \infty$.  Assume $X$ has property $\mathfrak{A}_{N-1,p}$. By Proposition \ref{facts}$(i)$, there exists $\alpha_0>0$ such that for any $n\in\nn$ and $\ee_1, \ldots, \ee_n>0$ such that $s^{\omega^{N-1}}_{\ee_1} \ldots s^{\omega^{N-1}}_{\ee_n}(B_{X^*})\neq \varnothing$, $\sum_{i=1}^n \ee_i^q \leqslant \alpha_0^q$.   Assume $A:\tim_{\ee,i=0}^{2N} Y_i \to X^*$ is an isomorphic embedding. Without loss of generality, we may assume that for some $c>0$, $c\|u\| \leqslant \|Au\|\leqslant \|u\|$ for all $u\in \tim_{\ee,i=0}^{2N-1} Y_i$.

Since $q<2$, it follows that $\frac{1}{q}-\frac{1}{2}>0$.   Fix $j$ so large that $j^{\frac{1}{q}-\frac{1}{2}} > 6\alpha_0/c$.   Combining Corollary \ref{brnsky}$(iv)$ with $\gamma\in(1,2)$ and Lemma \ref{nukemhi}$(ii)$ applied with $Y=\tim_{\ee,i=1}^{2N-1} Y_i$, there exist a tree $S$ with $\text{rank}(S)=\omega^Nj$ and a weakly null collection $(v_s)_{s\in S}$ such that, with $v_\varnothing=0$, $\|\sum_{r\preceq s} v_r\|\leqslant 1$ and $\|v_s\|> 1/6j^{1/2}$ for all $s\in S$.   Note that $(Av_s)_{s\in S}$ is weakly null, and therefore weak$^*$-null, in $X^*$.  Moreover, $\sum_{r\preceq s}Av_r\in B_{X^*}$ and $\|Av_s-v_{s^-}\|>c/6j^{1/2}$ for all $s\in S$.    From this and a similar induction argument to that in $(ii)$, it follows that for each ordinal $\zeta$ and each $s\in (\{\varnothing\}\cup S)^\zeta$, $Av_\varnothing+\sum_{r\preceq s} Av_r\in s^\zeta_{c/6j^{1/2}}(B_{X^*})$. In particular, since $\varnothing\in (\{\varnothing\}\cup S)^{\omega^{N-1}j}$, \[0=Av_\varnothing\in s^{\omega^{N-1}j}_{c/6j^{1/2}}(B_{X^*})= s^{\omega^{N-1}}_{c/6j^{1/2}} \ldots s^{\omega^{N-1}}_{c/6j^{1/2}}(B_{X^*}).\]   Therefore \[j^{\frac{1}{q}-\frac{1}{2}}\cdot c/6= \Bigl(\sum_{i=1}^j(c/6j^{1/2})^q\Bigr)^{1/q} \leqslant \alpha_0.\] This contradicts our choice of $j$ and finishes $(iv)$.

\end{proof}

The following facts were shown in the preceding proof, but we isolate them separately.

\begin{corollary} Fix $N\in\nn$. \begin{enumerate}[label=(\roman*)]\item If $Y_1, \ldots, Y_{2N-1}$ are infinite dimensional Banach spaces, then $Sz(\tim_{\pi,i=1}^{2N-1}Y_i)>\omega^{N-1}$. \item If $Y_1, \ldots, Y_{2N}$ are infinite dimensional Banach spaces, then $\tim_{\pi,i=1}^{2N}Y_i$ does not have $\mathfrak{A}_{N-1,p}$ for any $2<p\leqslant \infty$.   \end{enumerate}

\end{corollary}

\begin{rem}\upshape
We also note that the hypothesis that the duals have some non-trivial cotype in Theorem \ref{upper1} cannot be relaxed. In the proof above, we used the fact that for infinite dimensional spaces $E\tim_\ee F$ has trivial cotype, so each additional factor of $Y_{2i-1}\tim_\pi Y_{2i}$ will multiply the Szlenk index of the tensor product by at least a factor of $\omega$. If any of the spaces $Y_i$ has a dual with trivial cotype, then the factor $Y_i^*$ by itself will multiply the Szlenk index of the tensor product by at least a factor of $\omega$.  Indeed, if $Y^*_0$ has trivial cotype, then $Sz(Y_0\tim_\pi Y_1)>\omega$ for any infinite dimensional $Y_1$. Arguing as above, we can then guarantee that for any other infinite dimensional spaces $Y_2, Y_3, \ldots$, the spaces $\tim_{\pi,i=0}^N Y_i$ have properties that get worse at least by half steps, in the sense that $\tim_{\pi,i=0}^{2N} Y_i$ fails $\mathfrak{A}_{N-1,p}$ for each $2<p\leqslant \infty$, and $\tim_{\pi,i=0}^{2N-1} Y_i$ has Szlenk index exceeding $\omega^{N-1}$. 
\end{rem}

\begin{rem}\upshape We note that in Theorem \ref{thm}$(i)$ and $(iii)$, the conclusion is that a certain tensor product is not isomorphic to any subspace of any quotient of a given space, while items $(ii)$ and $(iv)$ said nothing about quotients.   Of course, for any Banach space $X$, $c_0(B_X)$ has property $\mathfrak{T}_{0,\infty}$, which, of the properties under investigation, is the strongest property an infinite dimensional Banach space may enjoy, while $X$ is isometrically isomorphic to a quotient of $\ell_1(B_X)$.  Moreover, $c_0$ enjoys $\mathfrak{T}_{0,\infty}$, while every separable Banach space is isometrically isomorphic to a quotient of $\ell_1$.  Therefore quotients cannot be included in items $(ii)$ and $(iv)$ of Theorem \ref{thm}.  Let us explain why quotients can be included in items $(i)$ and $(iii)$ but not in $(ii)$ and $(iv)$. 

For a Banach space $Y$, showing that $Sz(Y)>\omega$ or that $Y$ fails property $\mathfrak{A}_{N,p}$ involves constructing weakly null trees in $B_Y$ such that certain linear combinations of every branch of the tree satisfies a particular lower estimate.   If $Y$ is (isomormorphic to) a quotient of some Banach space $X$, these weakly null trees in $Y$ can be pulled back to weakly null trees in $X$. Since the linear combinations of the branches of the trees in $Y$ satisfy a lower estimate, the corresponding linear combinations of branches of the pulled back tree in $X$ satisfy a lower inequality, with possibly a scaled constant.  Therefore $X$ also fails the given property.   By contraposition, the properties $Sz(\cdot)\leqslant \omega^N$ and $\mathfrak{A}_{N,p}$ pass to quotients.    

However, in order to prove Theorem \ref{thm}$(ii)$, we showed that if $Sz(X)\leqslant \omega^N$, then $\tim_{\ee,i=0}^{2N} Y_i$ is not isomorphic to a subspace of $X^*$ by exhibiting $\ee>0$ and a weakly null (and therefore weak$^*$-null) tree $(x^*_t)_{t\in T}\subset B_{X^*}$ such that, with $x^*_\varnothing=0$, $\|x^*_t-x^*_{t^-}\|>\ee$  and $\|\sum_{s\preceq t} x^*_s\|\leqslant 1$ for all $t\in T$.  We used this to prove that for $\zeta$ and $t\in T^\zeta$, $\sum_{s\preceq t} x^*_s\in s^\zeta_\ee(B_{X^*})$.   However, this does not preclude such a weakly null tree existing in some quotient of $X^*$.    If $Y$ is some quotient of $X^*$ (or a non-dual Banach space $Z$) and if $(y_t)_{t\in T}$ is a weakly null collection in $Y$ such that $\|y_t-y_{t^-}\|>\ee$ and $\|\sum_{s\preceq t}y_s\|\leqslant 1$ for all $t\in T$, then we can pull the collection $(y_t)_{t\in T}$ back to a collection in $X^*$ (or, more generally, to $Z$). However, the boundedness of the quotient map does not provide any analogue of the inequality $\|\sum_{s\preceq t}y_s\|\leqslant 1$ in the pulled back collection.

\end{rem}

\section{Extension to the symmetric case}

\begin{lemma}Let $Y$ be an infinite dimensional  Banach space and fix $N\in\nn$. Let $T$ be a well-founded tree and assume $(\tim_{j=1}^N y_{j,t})_{t\in T}\subset \tim_\pi^N Y$, $(\tim_{j=1}^N \psi_{j,t})_{t\in T}\subset \tim_\ee^N Y^*$ are such that $\langle \psi_{j,t}, y_{j,t}\rangle=1$ for all   $t\in T$, and for permutations $\sigma, \tau$ on $\{1, \ldots, N\}$, $1\leqslant j\leqslant N$, and $t\sim t'$,  \begin{displaymath}
   \langle \otimes_{j=1}^N\psi_{\sigma(j),t'}, \otimes_{j=1}^Ny_{\tau(j),t}\rangle = \left\{
     \begin{array}{ll}
       1 & : \sigma = \tau, t=t'\\
       0 & : \text{otherwise.}
     \end{array}
   \right.
\end{displaymath} 

\begin{enumerate}[label=(\roman*)]\item Suppose  $\|y_{j,t}\|\leqslant 1$ for all $t\in T$ and $1\leqslant j\leqslant N$  and $\gamma>0$ is such that  $\|(\otimes_{j=1}^N \psi_{j,s})_{s\preceq t}\|_{\ee,1}^w  \leqslant \gamma$ for all $t\in MAX(T)$. Then there exist a tree $S$ with $\text{rank}(S)=\text{rank}(T)$ and  collections $(y_{j,s})_{s\in S}\subset B_Y$,  $(\psi_{j,s})_{s\in S}\subset Y^*$, $0\leqslant j\leqslant N$, such that \begin{enumerate}[label=(\alph*)]\item $\|\psi_{0,s}\|\leqslant 3$ for all $s\in S$,  \item  $(\otimes_{j=0}^N y_{j,s})_{s\in S}$ is weakly null in $\tim_\pi^{N+1}Y$, \item $\langle \psi_{j,s}, y_{j,s}\rangle = 1 $ for all $s\in S$ and $0\leqslant j\leqslant N$, \item for all $0\leqslant j\leqslant N$, $s,s'\in S$, and permutations $\sigma, \tau$ on $\{0, \ldots, N\}$, $0\leqslant j\leqslant N$, \begin{displaymath}
   \langle \otimes_{j=0}^N \psi_{\sigma(j),t'}, \otimes_{j=0}^N y_{\tau(j),t}\rangle = \left\{
     \begin{array}{ll}
       1 & : \sigma = \tau, t=t'\\
       0 & : \text{otherwise,}
     \end{array}
   \right.
\end{displaymath} and \item $\|(\otimes_{j=0}^N \psi_{j,r})_{r\preceq s}\|_{\ee,1}^w \leqslant 3\gamma$ for all $s\in S$.  \end{enumerate}

\item If $\|\psi_{j,t}\|\leqslant 1$ for all $t\in T$ and $1\leqslant j\leqslant N$ and $\gamma>0$ is such that $\| (\otimes_{j=1}^N y_{j,s})_{s\preceq t}\|_{\ee,1}^w\leqslant \gamma$ for all $t\in MAX(T)$, then there exist collections $(y_{j,s})_{s\in S}\subset B_Y$,  $(\psi_{j,s})_{s\in S}\subset Y^*$, $0\leqslant j\leqslant N$, satisfying (a)-(d) above, and  \begin{enumerate}[label=(\alph*$'$), start=5]\item $\|(\otimes_{j=0}^N y_{j,r})_{r\preceq s}\|_{\ee,1}^w \leqslant \gamma$ for all $s\in S$. \end{enumerate} \end{enumerate}\label{2kemhi}\end{lemma}

\begin{proof}

\end{proof} We argue exactly as in Lemma \ref{nukemhi}.  Let $D$ be the set of finite codimensional subspaces of $Y$, directed by reverse inclusion.   Let \[S=\{(\nu_i, E_i)_{i=1}^n:(\nu_i)_{i=1}^n\in T, E_1, \ldots, E_n\in D\}.\]  For $s=(\nu_i, E_i)_{i=1}^n\in S$, we let $t_s=(\nu_i)_{i=1}^n\in T$.     We define collections $(y_{0,s})_{s\in S}\subset B_Y$, $(\psi_{0,s})_{s\in S}\subset 3B_{Y^*}$ by recursion on $|s|$ such that, with $y_{j,s}=y_{j,t_s}$ and $\psi_{j,s}=\psi_{j,t_s}$ for all $s\in S$ and $1\leqslant j\leqslant N$,  $(\otimes_{j=0}^N y_{0,s})_{s\in S}$ is weakly null in $\tim_\pi^{N+1}Y$, $\langle \psi_{0,s}, y_{0,s}\rangle=1$ for all $s\in S$, and $\langle \psi_{0,s}, y_{j,s'}\rangle = \langle \psi_{j,s'}, y_{0,s}\rangle=0$ for all $0\leqslant j\leqslant N$ and $s'\prec s\in S$.   Therefore $(a)$ and $(b)$ are satisfied.  Since $y_{j,s}=y_{j, t_s}$ and $\psi_{j,s}=\psi_{j,_s}$ for all $1\leqslant j\leqslant N$ and $s\in S$, $(c)$ is satisfied. Condition $(e)$ in $(i)$ follows from the fact that $\|\psi_{0,s}\|\leqslant 3$ for all $s\in S$, the hypothesis that $\|(\otimes_{j=1}^N \psi_{j,s})_{s\preceq t}\|_1^w\leqslant \gamma$ for all $t\in T$, and  Proposition \ref{pupper}$(i)$.  Condition $(e')$ in $(ii)$ is derived similarly. 

We last show the condition $(d)$.  Fix $s,s'\in S$ with $s\sim s'$ and permutations $\sigma, \tau$ on $\{0, \ldots, N\}$.  It follows from the construction that $\langle \otimes_{j=0}^N \psi_{\sigma(j), s'}, \otimes_{j=0}^N y_{\tau(j), s}\rangle=1$ if $s=s'$ and $\sigma=\tau$.  If $j_0:=\sigma^{-1}(0)\neq \tau^{-1}(0)=:j_1$, then $\langle \otimes_{j=0}^N \psi_{\sigma(j), s'}, \otimes_{j=0}^N y_{\tau(j),s}\rangle$ includes a factor of $\langle \psi_{0,s'}, y_{\tau(j_0), s}\rangle$, which is zero if $s\preceq s'$, since $1\leqslant \tau(j_0)\leqslant N$, and a factor of $\langle \psi_{\sigma(j_1), s'}, y_{0, s}\rangle$, which is zero if $s'\preceq s$, since $1\leqslant \sigma(j_1)\leqslant N$.     Therefore $\langle \otimes_{j=0}^N \psi_{\sigma(j), s'}, y_{\tau(j), s}\rangle=0$ if $\sigma(0)\neq 0$ or if $\tau(0)\neq 0$.   Assume $\sigma(0)=0$ and $\tau(0)=0$.  Then $\sigma|_{\{1, \ldots, N\}}, \tau_{\{1, \ldots, N\}}$ are permutations on $\{1, \ldots, N\}$. By hypothesis, and since $s\sim s'$ implies $t_s\sim t_{s'}$, and $s\not\sim s'$ if and only if $t_s\not\sim t_{s'}$, \[\langle \otimes_{j=1}^N \psi_{\sigma(j),s'}, \otimes_{j=1}^N y_{\tau(j), s}\rangle\] will equal $0$ if $\sigma_{\{1, \ldots, N\}}\neq \tau_{\{1, \ldots, N\}}$ or if $s\sim s'$. Since we are in the $\sigma(0)=0=\tau(0)$ case, $\sigma|_{\{1, \ldots, N\}}\neq \tau|_{\{1, \ldots, N\}}$ if and only if $\sigma\neq \tau$. Since $\langle \otimes_{j=0}^N \psi_{\sigma(j), s'}, \otimes_{j=0}^N y_{\tau(j), s}\rangle$ includes $\langle \otimes_{j=1}^N \psi_{\sigma(j), s'}, \otimes_{j=1}^N y_{\tau(j), s}\rangle$ as a factor in the $\sigma(0)=0=\tau(0)$ case, we deduce that $\langle \otimes_{j=0}^N \psi_{\sigma(j), s'}, \otimes_{j=0}^N y_{\tau(j), s}\rangle =0 $ if $\sigma\neq \sigma'$ or $s\not\sim s'$.

We next isolate the following easy consequence of Dvoretzky's theorem. 

\begin{lemma} Fix $1\leqslant p,q\leqslant \infty$ with $1/p+1/q=1$.   \begin{enumerate}[label=(\roman*)]\item If $Y$ is a  Banach space in which $\ell_p$ is finitely representable, $E\subset Y$ is a finite set, and $F\subset Y^*$ is a finite set, then for any $n\in\nn$, any $\beta>2$, and  $\gamma>1$, there exist $(y_i)_{i=1}^n\subset \cap_{\psi\in F}\ker(\psi)$, $(\psi_i)_{i=1}^n \subset \cap_{y\in E} \ker(y)$ such that $\langle \psi_i, y_j\rangle=\delta_{i,j}$, $\|(\psi_i)_{i=1}^n\|_\infty\leqslant \beta$, and $\|(y_i)_{i=1}^n\|_q^w\leqslant \gamma$. 

 \item If $Y$ is a  Banach space such that  $\ell_p$ is finitely representable in $Y^*$, $E\subset Y$ is a finite set, and $F\subset Y^*$ is a finite set, then for any $n\in\nn$ and $\gamma>2$, there exist $(y_i)_{i=1}^n\subset \cap_{\psi\in F}\ker(\psi)$, $(\psi_i)_{i=1}^n\subset  \cap_{y\in E} \ker(y)$ such that $\langle \psi_i, y_j\rangle=\delta_{i,j}$, $\|(y_i)_{i=1}^n\|_\infty\leqslant \beta$, and $\|(\psi_i)_{i=1}^n\|_q^w\leqslant \gamma$.

\end{enumerate}
\label{tail}
\end{lemma}

\begin{proof}$(i)$  Fix $\beta>2$, $\gamma>1$, $E\subset Y$ finite, and $F\subset Y^*$ finite.  Let $E_1=\text{span}\{E\}$.   Fix a finite subset $G$ of $B_{Y^*}$ such that for each $y\in E_1$, $(\beta-1)\sup_{\psi\in G} |\langle \psi, y\rangle|\geqslant \|y\|$ and let $Y_1=\cap_{\psi\in F\cup G} \ker(\psi)$.  Note that $E_1$ is $(\beta-1)$-complemented in $E_1\oplus Y_1$. Indeed, for $y\in E_1$, if $\psi\in G$ is such that $\|y\|\leqslant (\beta-1) |\langle \psi, y\rangle|$, then for any $y_2\in Y_1$, \[(\beta-1)\|y+y_2\| \geqslant (\beta-1)|\langle \psi, y+y_2\rangle|=(\beta-1)|\langle \psi, y\rangle|\geqslant \|y\|.\] Since $E_1$ is $(\beta-1)$-complemented in $E_1\oplus Y_1$,  $Y_1$ is $\beta$-complemented in $E_1\oplus Y_1$. Let $P:E_1\oplus Y_1\to Y_1$ be a projection with $\|P\|\leqslant \beta$. 

 Since $\ell_p$ is finitely representable in $Y$, $\ell_p$ is finitely representable in $Y_1$. Therefore for any $n\in\nn$, we can find $(y_i)_{i=1}^n \subset Y_1$  and $(\xi_i)_{i=1}^n\subset B_{Y^*_1}$ such that $\langle \xi_i, y_j\rangle=\delta_{i,j}$ and such that $\|(y_i)_{i=1}^n\|_q^w\leqslant \gamma$.  Let $\psi_i$ be a Hahn-Banach extension of the functional $P^*\xi_i$, from which it follows that $\|\psi_i\|\leqslant \beta$.

$(ii)$ Fix $\beta>2$, $\gamma>1$, $E\subset Y$ finite, and $F\subset Y^*$ finite. Fix $\beta_0\in (2,\beta)$.   Let $F_1=\text{span}\{F\}$.   Fix a finite subset $G$ of $B_{Y}$ such that for each $\psi\in F_1$, $(\beta_0-1)\sup_{y\in G} |\langle \psi, y\rangle|\geqslant \|\psi\|$ and let $Z=\cap_{\psi\in E\cup G} \ker(y)$.  Note that $F_1$ is $(\beta_0-1)$-complemented in $F_1\oplus Z$. Let $P:F_1\oplus Z\to Z$ be a projection with $\|P\|\leqslant \beta_0$. 

 Since $\ell_p$ is finitely representable in $Y^*$, $\ell_p$ is finitely representable in $Z$. Therefore for any $n\in\nn$, we can find $(\xi_i)_{i=1}^n \subset Z$  and $(z^*_i)_{i=1}^n\subset B_{Z^*}$ such that $\langle z^*_j, \xi_i\rangle=\delta_{i,j}$ and such that $\|(\xi_i)_{i=1}^n\|_q^w\leqslant \gamma$.  Let $y^{**}_i$ be a Hahn-Banach extension of the functional $P^*z^*_i$, from which it follows that $\|y^{**}_i\|\leqslant \beta_0$. By local reflexivity, we can choose $(y_i)_{i=1}^n\subset Y$ such that for each $1\leqslant i\leqslant n$ and $\psi\in \{\psi_1, \ldots, \psi_n\}\cup F$, $\|y_i\|\leqslant \beta$ and  $\langle y^{**}_i, \psi\rangle = \langle \psi, y_i\rangle$.

\end{proof}

\begin{lemma} Fix $N\in\nn$ and an infinite dimensional Banach space $Y$.  Fix finite subsets $E\subset Y$ and $F\subset Y^*$ and let $Y_0=\cap_{\psi\in F}\ker(\psi)\subset Y$, $Z=\cap_{y\in E}\ker(y)\subset Y^*$.  For any $\gamma>1$, $\beta>2$, and $N,n\in\nn$, there exist a tree $T$ with $\text{rank}(T)=\omega^{\lfloor \frac{N-1}{2}\rfloor}n$ and collections $(y_{j,t})_{t\in T}, (z_{j,t})_{t\in T}\subset Y_0$, $(\psi_{j,t})_{t\in T}, (\zeta_{j,t}){t\in T}\subset Z$, $1\leqslant j\leqslant N$,  such that \begin{enumerate}[label=(\roman*)]\item $\langle \psi_{j,t}, y_{j,t}\rangle=\langle \zeta_{j,t}, z_{j,t}\rangle= 1$ for all $1\leqslant j\leqslant N$ and $t\in T$,  \item for permutations $\sigma, \tau$ on $\{1, \ldots, N\}$, $1\leqslant j\leqslant N$, and $t\sim t'$,  \begin{displaymath}
   \langle \otimes_{j=1}^N\psi_{\sigma(j),t'}, \otimes_{j=1}^Ny_{\tau(j),t}\rangle = \left\{
     \begin{array}{ll}
       1 & : \sigma = \tau, t=t'\\
       0 & : \text{otherwise,}
     \end{array}
   \right.
\end{displaymath} and \begin{displaymath}
   \langle \otimes_{j=1}^N\zeta_{\sigma(j),t'}, \otimes_{j=1}^Nz_{\tau(j),t}\rangle = \left\{
     \begin{array}{ll}
       1 & : \sigma = \tau, t=t'\\
       0 & : \text{otherwise,}
     \end{array}
   \right.
\end{displaymath}\item for all $1\leqslant j\leqslant N$ and $t\in T$, $\|y_{j,t}\|, \|\zeta_{j,t}\|\leqslant \beta$,  \item if $N$ is even, then for all $t\in T$, $\|(\otimes_{j=1}^{N} \psi_{j,s})_{s\preceq t}\|_{\ee,1}^w, \|(\otimes_{j=1}^{N}z_{j,s})_{s\preceq t}\|_{\ee,1}^w\leqslant \gamma$, \item if $N$ is odd, then for all $t\in T$, $\|(\otimes_{j=1}^{N} \psi_{j,s})_{s\preceq t}\|_{\ee,1}^w, \|(\otimes_{j=1}^{N}z_{j,s})_{s\preceq t}\|_{\ee,1}^w\leqslant \gamma n^{1/2}$, \item  \end{enumerate}

Moreover, if $N$ is even, then the tree $T$ can be taken to have rank $\omega^{N/2}$.

\label{mltdwn}
\end{lemma}

\begin{proof} As in Lemma \ref{subhmnd}, if $N$ is even, say $N=2M$, the ``moreover'' statement at the end of the proof follows from the preceding conclusions by taking a totally incomparable union of trees $T_1, T_2, \ldots$ with $\text{rank}(T_n)=\omega^{M}n$. 

Many of the details are checked as in the proof of Lemma \ref{subhmnd}.  We work by induction.    We begin with the $N=1$ and $N=2$ cases.   The $N=1$ case is a direct application of Lemma \ref{tail} applied with $p=2$, using Dvoretzy's theorem. We let \[T=\{(n-1, \ldots, n-i): 1\leqslant i\leqslant n\},\] which has rank $n=\omega^0 n$.  We choose $(y_i)_{i=1}^n$ and $(\psi_i)_{i=1}^n$ according to Lemma \ref{tail}$(ii)$ and $(z_i)_{i=1}^n, (\zeta_i)_{i=1}^n$ according to Lemma \ref{tail}$(i)$.   For $t=(n-1, \ldots, n-i)$, we let $y_{1,t}=y_i$, $z_{1,t}=z_i$, $\psi_{1,t}=\psi_i$, and $\zeta_{1,t}=\zeta_i$.   Conditions (i) and (ii) follow from biorthogonality. Here we note that there is only one permutation on $\{1\}$. 

For $N=2$ and $n\in\nn$, again by Lemma \ref{tail} with $p=2$, we can find $(y_{1,i})_{i=1}^n, (z_{1,i})_{i=1}^n\subset Y_0$ and $(\psi_{1,i})_{i=1}^n, (\zeta_{1,i})_{i=1}^n\subset Z$ such that $\|(\psi_{1,i})_{i=1}^n\|_2^w, \|(z_{1,i})_{i=1}^n\|_2^w\leqslant \gamma^{1/2} $,  $\|y_{1,i}\|, \|\zeta_{1,i}\|\leqslant \beta$, and $\langle \psi_{1,i}, y_{1,j}\rangle=\langle \zeta_{1,i}, z_{1,j}\rangle=\delta_{i,j}$ for all $1\leqslant i,j\leqslant n$.    We then let $E'=E\cup \{y_{1,1}, \ldots, y_{1,n}, z_{1,1}, \ldots, z_{1,n}\}$ and $F'=F\cup \{\psi_{1,1}, \ldots, \psi_{1,n}, \zeta_{1,1}, \ldots, \zeta_{1,n}\}$, $Y_1=\cap_{\psi\in F'} \ker(\psi)$ and $Z_1=\cap_{y\in E'} \ker(y)$.  We then apply Lemma \ref{tail} again to find $(y_{2,i})_{i=1}^n, (z_{2,i})_{i=1}^n\subset Y_1$ and $(\psi_{2,i})_{i=1}^n, (\zeta_{2,i})_{i=1}^n\subset Z_1$ such that $\|(\psi_{2,i})_{i=1}^n\|_2^w, \|(z_{2,i})_{i=1}^n\|_2^w\leqslant \gamma^{1/2} $,  $\|y_{2,i}\|, \|\zeta_{2,i}\|\leqslant \beta$, and $\langle \psi_{2,i}, y_{2,j}\rangle=\langle \zeta_{2,i}, z_{2,j}\rangle=\delta_{i,j}$ for all $1\leqslant i,j\leqslant n$.     Again, let \[T=\{(n-1, \ldots, n-i): 1\leqslant i\leqslant n\}.\]   For $t=(n-1, \ldots, n-i)$ and $k\in \{1,2\}$, let $y_{k,t}=y_{k,i}$, $z_{k,t}=z_{k,i}$, $\psi_{k,t}=\psi_{k,i}, \zeta_{k,t}=\zeta_{k,i}$. Items $(i)$ and $(iii)$ are immediate from the construction. We note that for any $t=(n-1, \ldots, n-i)\in T$, by Proposition \ref{pupper}$(i)$, \[\|(\otimes_{j=1}^2 \psi_{j,s})_{s\preceq t}\|_{\ee,1}^w= \|(\otimes_{j=1}^2 \psi_{j,l})_{l=1}^i \|_{\ee,1}^w \leqslant \|(\psi_{1,l})_{l=1}^n \|_2^w \|(\psi_{2,l})_{l=1}^n\|_1^w \leqslant \gamma.\]    We deduce that $\|(\otimes_{j=1}^2 z_{j,s})_{s\preceq t}\|_{\ee,1}^w \leqslant \gamma$ similarly.   This yields $(iv)$.       Let us verify $(ii)$.     Note that since $y_{2,i}, z_{2,i}\in Y_1$ for each $1\leqslant i\leqslant n$, $\langle \psi_{1,j}, y_{2,i}\rangle=\langle \zeta_{1,j}, z_{2,i}\rangle=0$ for any $1\leqslant j\leqslant n$.  Similarly, $\langle \psi_{2,i}, y_{1,j}\rangle = \langle \zeta_{2,i}, z_{1,j}\rangle=0$ for all $1\leqslant i,j\leqslant n$.  From this it follows that if $t\sim t'$ and $\sigma\neq \tau$ are permutations on $\{1,2\}$, then \[\langle \otimes_{j=1}^2 \psi_{\sigma(j), t'}, \otimes_{j=1}^2 y_{\tau(j), t}\rangle=\langle \otimes_{j=1}^2 \zeta_{\sigma(j), t'}, \otimes_{j=1}^2 z_{\tau(j), t}\rangle=0.\] This is because if $i=|t|$ and $i'=|t'|$, $\langle \otimes_{j=1}^2 \psi_{\sigma(j), t'}, \otimes_{j=1}^2 y_{\tau(j), t}\rangle$ has a factor of $\langle \psi_{2,i'}, y_{1,i}\rangle=0$ and $\langle \otimes_{j=1}^2 \zeta_{\sigma(j), t'}, \otimes_{j=1}^2 z_{\tau(j), t}\rangle$ has a factor of $\langle \zeta_{2,i'}, z_{1,i}\rangle=0$. 

We now assume the result holds for $N=2M$ and deduce the result for $N=2M+1$.   Fix $\beta>2$, $\gamma>1$, and $n\in\nn$. We select $(y_{2M+1,i})_{i=1}^n, (z_{2M+1,i})_{i=1}^n\subset Y_0$, $(\psi_{2M+1,i})_{i=1}^n, (\zeta_{2M+1,i})_{i=1}^n\subset Z$ as in the $N=1$ case with $\gamma$ replaced by $\gamma^{1/2}$.    We then define $E'=E\cup \{ y_{2M+1,1}, \ldots, y_{2M+1,n}, z_{2M+1,1}, \ldots, z_{2M+1,n}\}$ and $F'=F\cup \{\psi_{2M+1,1}, \ldots, \psi_{2M+1,n}, \zeta_{2M+1, 1}, \ldots, \zeta_{2M+1,n}\}$.    Let $Y_1=\cap_{\psi\in F'} \ker(\psi)$ and $Z_1= \cap_{y\in E'} \ker(y)$.   By the inductive hypothesis and the ``moreover'' statement from the end of the proof, there exist a tree $R$ with $\text{rank}(R)=\omega^M$ and collections $(y'_{j,r})_{r\in R}, (z'_{j,r})_{j\in R}\subset Y_1$, $(\psi'_{j,r})_{r\in R}, (\zeta'_{j,r})_{r\in R}\subset Z_1$, $1\leqslant j\leqslant 2M$, satisfying the conclusions of the lemma with $\gamma$ replaced by $\gamma^{1/2}$.    As in Lemma \ref{subhmnd}, let $T$ denote the set of all sequences of pairs $(r_i, \nu_i)_{i=1}^k$ such that there exist $1\leqslant m\leqslant n$ and  an interval partition $I_{n-1}, \ldots, I_{n-m}$ of $\{1, \ldots, k\}$ such that \begin{enumerate}[label=(\alph*)]\item $I_{n-1}<\ldots <I_{n-m}$, \item for each $1\leqslant l\leqslant m$, $r_i=n-l$ for each $i\in I_{n-l}$, \item for each $1\leqslant l\leqslant m$, $(\nu_i)_{i\in I_{n-m}}\in R$. \end{enumerate} Again, we refer to this $m$ as the \emph{outer length} of the sequence.   For $t=(r_i, \nu_i)_{i=1}^k\in T$ with outer length $m$ and interval partition $I_{n-1}, \ldots, I_{n-m}$, we let $y_{2M+1,t}=y_{2M+1,m}$, $z_{2M+1,t}=z_{2M+1,m}$, $\psi_{2M+1,t}=\psi_{2M+1,m}$, and $\zeta_{2M+1,t}=\zeta_{2M+1,m}$.  For $1\leqslant j\leqslant 2M$, we let $y_{j,t}=y_{j,(\nu_i)_{i\in I_{n-m}}}$, $z_{j,t}=z_{j,(\nu_i)_{i\in I_{n-m}}}$, $\psi_{j,t}=\psi_{j,(\nu_i)_{i\in I_{n-m}}}$, and $\zeta_{j,t}=\zeta_{j,(\nu_i)_{i\in I_{n-m}}}$.   Items $(i)$ and $(iii)$ are easily verified.    Item $(v)$ is verified exactly as in the proof of Lemma \ref{subhmnd}.     

We verify $(ii)$.   We must fix permutations $\sigma, \tau$ on $\{1, \ldots, 2M+1\}$ and $t\sim t'\in T$ such that either $\sigma\neq \tau$ or $t\not\sim t'$ and prove that $\langle \otimes_{j=1}^{2M+1} \psi_{\sigma(j), t'}, \otimes_{j=1}^{2M+1}y_{\tau(j), t}\rangle = 0$. The argument that $\langle \otimes_{j=1}^{2M+1} \zeta_{\sigma(j), t'}, \otimes_{j=1}^{2M+1}z_{\tau(j), t}\rangle = 0$ is identical. Note that \begin{align*}\langle \otimes_{j=1}^{2M+1} \psi_{\sigma(j), t'}, y_{\tau(j), t}\rangle & = \prod_{j=1}^{2M+1} \langle \psi_{\sigma(j), t'}, y_{\tau(j), t}\rangle = \prod_{j=1}^{2M+1} \langle \psi_{\sigma\circ \tau^{-1}(j), t'}, y_{\tau\circ \tau^{-1}(j), t}\rangle \\ & = \prod_{j=1}^{2M+1} \langle \psi_{\sigma\circ \tau^{-1}(j), t'}, y_{j,r}\rangle. \end{align*}  Therefore it is sufficient to consider the case $\tau$ is the identity permutation $\iota_{2M+1}$ and either $\sigma\neq \iota_{2M+1}$ or $t\not\sim t'$.   We will proceed in the following cases: \begin{enumerate}[label=(\alph*)]\item $\sigma(2M+1)\neq 2M+1$. \item $\sigma(2M+1)=2M+1$ and $t,t'$ have different outer length. \item $\sigma(2M+1)=2M+1$ and $t,t'$ have the same outer length. \end{enumerate}  

(a) Note that since $(y_{j,r})_{r\in R}\subset Y_1$ and $(\psi_{j,r})_{r\in R} \subset Z_1$, $\langle \psi_{2M+1,t'}, y_{j,t}\rangle=\langle \psi_{j,t'}, y_{2M+1,t}\rangle=0$ for any $1\leqslant j\leqslant 2M$.  Therefore if $\sigma(2M+1)\neq 2M+1$, $\langle \otimes_{j=1}^{2M+1} \psi_{\sigma(j), t'}, \otimes_{j=1}^{2M+1}y_{j, t}\rangle = 0$, since this quantity includes a factor of $\langle \psi_{\sigma(2M+1),t'}, y_{2M+1, t}\rangle=0$ and $1\leqslant \sigma(2M+1)\leqslant 2M$. 

(b) Assume $\sigma(2M+1)=2M+1$ and let $\sigma'=\sigma|_{\{1, \ldots, 2M\}}$. Note that $\sigma'$ is a permutation on $\{1, \ldots, 2M\}$ and $\sigma= \iota_{2M+1}$   if and only if $\sigma'=\iota_{2M}$.     Assume $t, t'$ have different outer lengths. If $t$ has outer length $l$ and $t'$ has outer length $l'$, $\psi_{\sigma(2M+1),t'}=\psi_{2M+1,t'}=\psi_{2M+1,l'}$ and $y_{2M+1, t}=y_{2M+1,l}$.  By biorthogonality of the sequences $(y_{2M+1,i})_{i=1}^n$ and $(\psi_{2M+1, i})_{i=1}^n$, \[\langle \psi_{\sigma(2M+1), t'}, y_{2M+1,t}\rangle = \langle \psi_{2M+1,l'}, y_{2M+1,l}\rangle=0.\] Since $\langle \otimes_{j=1}^{2M+1} \psi_{\sigma(j), t'}, \otimes_{j=1}^{2M+1} y_{j,t}\rangle=0$, since it has a factor of $\langle \psi_{\sigma(2M+1), t'}, y_{2M+1, t}\rangle=0$.

(c) Assume $\sigma(2M+1)=2M+1$ and $t,t'$ have outer length $l$. Let $\sigma'$ be as in the previous paragraph. Since $t\tim t'$, there exist $k,k'$ and a sequence $(r_i, \nu_i)_{i=1}^{\{\max k, k'\}}$ such that $t=(r_i, \nu_i)_{i=1}^k$ and $t'=(r_i, \nu_i)_{i=1}^{k'}$. Furthermore, there exist intervals $I,J$ such that the interval partition of $\{1, \ldots, k\}$ associated with $t$ is $I_{n-1}, \ldots, I_{n-l+1}, I$ and the interval partition associated with $t'$ is $I_{n-1}, \ldots, I_{n-l+1}, J$ for some $I_{n-1}, \ldots, I_{n-l+1}$.     Define $r=(\nu_i)_{i\in I}\in T$ and $r'=(\nu_i)_{i\in J}$ and note that $r\sim r'$. More precisely, $r\prec r'$ if $k<k'$, $r=r'$ if $k=k'$, and $r'\prec r$ if $k'<k$.    From this it follows that $r\sim r'$, and $r\not\sim r'$ if and only if $t\not \sim t'$.    Moreover, for each $1\leqslant j\leqslant 2M$, $\psi_{\sigma(j), t'}=\psi_{\sigma'(j), r'}$ and $y_{j,t}=y_{j,r}$.   By the properties of $(y_{j,r})_{r\in R}$ and $(\psi_{j,r})_{r\in R}$, $1\leqslant j\leqslant 2M$, and since either $r\not\sim r'$ or $\sigma'\neq \iota_{2M}$, which follows from the assumption that either $t\not\sim t'$ or $\sigma\neq \iota_{2M+1}$ and $\sigma(2M+1)=2M+1$, it follows that \[\langle \otimes_{j=1}^{2M} \psi_{\sigma(j), t'}, \otimes_{j=1}^{2M} y_{j, t}\rangle = \langle \otimes_{j=1}^{2M} \psi_{\sigma'(j), r'}, y_{j, r}\rangle =0.\]  Since $\langle \otimes_{j=1}^{2M+1}\psi_{\sigma(j), t'}, \otimes_{j=1}^{2M+1} y_{j,t}\rangle$ includes a factor of $\langle \otimes_{j=1}^{2M} \psi_{\sigma(j), t'}, y_{j,t}\rangle$, $\langle \otimes_{j=1}^{2M+1}\psi_{\sigma(j), t'}, \otimes_{j=1}^{2M+1} y_{j,t}\rangle=0$.

We now assume the result holds for $N=2M$ and deduce the result for $N=2M+2$.   Fix $\beta>2$, $\gamma>1$, and $n\in\nn$. We select $(y_{2M+1,i})_{i=1}^n, (z_{2M+1,i})_{i=1}^n\subset Y_0$, $(\psi_{2M+1,i})_{i=1}^n, (\zeta_{2M+1,i})_{i=1}^n\subset Z$ the same way $(y_{1,i})_{i=1}^n$, $(z_{1,i})_{i=1}^n$ and $(\psi_{1,i})_{i=1}^n, (\zeta_{1,i})_{i=1}^n$ were selected in the $N=2$ case with $\gamma^{1/2}$ replaced by $\gamma^{1/3}$.      We then define $E'=E\cup \{ y_{2M+1,1}, \ldots, y_{2M+1,n}, z_{2M+1,1}, \ldots, z_{2M+1,n}\}$ and $F'=F\cup \{\psi_{2M+1,1}, \ldots, \psi_{2M+1,n}, \zeta_{2M+1, 1}, \ldots, \zeta_{2M+1,n}\}$  and let $Y_1=\cap_{\psi\in F'} \ker(\psi)$ and $Z_1= \cap_{y\in E'} \ker(y)$. We then select $(y_{2M+2,i})_{i=1}^n$, $(z_{2M+2,i})_{i=1}^n \subset Y_1$ and $(\psi_{2M+2,i})_{i=1}^n, (\zeta_{2M+2,i})_{i=1}^n\subset Z_1$  the same way $(y_{2,i})_{i=1}^n$, $(z_{2,i})_{i=1}^n , (\psi_{2,i})_{i=1}^n, (\zeta_{2,i})_{i=1}^n$ were selected in the $N=2$ case with $\gamma^{1/2}$ replaced by $\gamma^{1/3}$.   Let $E''=E'\cup \{y_{2M+2,1}, \ldots, y_{2M+2,n}, z_{2M+2,1}, \ldots, z_{2M+2,n}\}$, $F''=F'\cup \{\psi_{2M+2,1}, \ldots, \psi_{2M+2,n}, \zeta_{2M+2,1}, \ldots, \zeta_{2M+2,n}\}$ and let $Y_2=\cap_{\psi\in F''}\ker(\psi)$ and $Z_2=\cap_{y\in E''} \ker(y)$.      By the inductive hypothesis and the ``moreover'' statement from the end of the proof, there exist a tree $R$ with $\text{rank}(R)=\omega^M$ and collections $(y'_{j,r})_{r\in R}, (z'_{j,r})_{j\in R}\subset Y_1$, $(\psi'_{j,r})_{r\in R}, (\zeta'_{j,r})_{r\in R}\subset Z_1$, $1\leqslant j\leqslant 2M$, satisfying the conclusions of the lemma with $\gamma$ replaced by $\gamma^{1/3}$.    As in Lemma \ref{subhmnd}, let $T$ denote the set of all sequences of pairs $(r_i, \nu_i)_{i=1}^k$ such that there exist $1\leqslant m\leqslant n$ and  an interval partition $I_{n-1}, \ldots, I_{n-m}$ of $\{1, \ldots, k\}$ such that \begin{enumerate}[label=(\alph*)]\item $I_{n-1}<\ldots <I_{n-m}$, \item for each $1\leqslant l\leqslant m$, $r_i=n-l$ for each $i\in I_{n-l}$, \item for each $1\leqslant l\leqslant m$, $(\nu_i)_{i\in I_{n-m}}\in R$. \end{enumerate} Again, we refer to this $m$ as the \emph{outer length} of the sequence.   For $t=(r_i, \nu_i)_{i=1}^k\in T$ with outer length $m$ and interval partition $I_{n-1}, \ldots, I_{n-m}$ and for $k\in \{2M+1, 2M+2\}$, let  $y_{k,t}=y_{k,m}$, $z_{k,t}=z_{k,m}$, $\psi_{k,t}=\psi_{k,m}$, and $\zeta_{k,t}=\zeta_{k,m}$.  For $1\leqslant j\leqslant 2M$, we let $y_{j,t}=y_{j,(\nu_i)_{i\in I_{n-m}}}$, $z_{j,t}=z_{j,(\nu_i)_{i\in I_{n-m}}}$, $\psi_{j,t}=\psi_{j,(\nu_i)_{i\in I_{n-m}}}$, and $\zeta_{j,t}=\zeta_{j,(\nu_i)_{i\in I_{n-m}}}$.   Items $(i)$ and $(iii)$ are easily verified.    Item $(iv)$ is verified exactly as in the proof of Lemma \ref{subhmnd}.

We verify $(ii)$ by showing that if $\sigma$ is a permutation on $\{1, \ldots, 2M+2\}$ and if $t,t'\in T$ are such that $t\sim t'$, and if either $\sigma \neq \iota_{2M+2}$ or $t\not\sim t'$, then $\langle \otimes_{j=1}^{2M+2} \psi_{\sigma(j), t'}, \otimes_{j=1}^{2M+2} y_{j,t}\rangle=0$.  Note that by the definitions of $E', E'', F', F''$ and our choices of vectors in either $Y_1$ or $Y_2$ and functionals in either $Z_1$ or $Z_2$, it follows that if  $k,l\in \{1, \ldots, 2M+2\}$ are distinct and at least one of $k,l$ is a member of $\{2M+1, 2M+2\}$, then $\langle \psi_{k,t'}, y_{l,t}\rangle=0$. Therefore if  $\sigma(2M+1)\neq 2M+1$, then with $l=2M+1$ and $k=\sigma(2M+1)$, $\langle \psi_{\sigma(2M+1), t'}, y_{2M+1,t}\rangle=0$, and $\langle \otimes_{j=1}^{2M+2} \psi_{\sigma(j), t'}, \otimes_{j=1}^{2M+2} y_{j,t}\rangle=0$.  Similarly, if $\sigma(2M+2)\neq 2M+2$, then $\langle \otimes_{j=1}^{2M+2}\psi_{\sigma(j), t'}, \otimes_{j=1}^{2M+2} y_{j,t}\rangle=0$.    Therefore $\langle \otimes_{j=1}^{2M+2} \psi_{\sigma(j), t'}, \otimes_{j=1}^{2M+2} y_{j,t}\rangle=0$ if either $\sigma(2M+1)\neq 2M+1$ or $\sigma(2M+2)\neq 2M+2$.   This is analogous to the case (a) above.    If we assume that $\sigma(2M+1)=2M+1$, $\sigma(2M+2)=2M+2$, and $t,t'$ have different outer lengths, then we deduce that $\langle \otimes_{j=1}^{2M+2} \psi_{\sigma(j), t'}, \otimes_{j=1}^{2M+2} y_{j,t}\rangle=0$ as in case (b) above. The remaining case, $\sigma(2M+1)=2M+1$,  $\sigma(2M+2)=2M+2$, and $t,t'$ have the same outer length, we deduce that $\langle \otimes_{j=1}^{2M+2} \psi_{\sigma(j), t'}, \otimes_{j=1}^{2M+2} y_{j,t}\rangle=0$ exactly as in case (c) above.

\end{proof}

\begin{theorem} Let $Y$ be an infinite dimensional Banach space. Let $X$ be a Banach space and fix $N\in\nn$. \begin{enumerate}[label=(\roman*)]\item If $Sz(X)\leqslant \omega^N$,  $\tim_{\pi,s}^{2N+1} Y$ is not isomorphic to any subspace of any quotient of $X$.  \item If $Sz(X)\leqslant \omega^N$, then $\tim_{\ee,s}^{2N+1} Y$ is not isomorphic to any subspace of $X^*$. \item  If $X$ has property $\mathfrak{A}_{N-1,\infty}$, then $\tim_{\pi,s}^{2N} Y$ is not isomorphic to any subspace of any quotient of $X$. \item If $X$ has property $\mathfrak{A}_{N-1,\infty}$, then $\tim_{\ee,s}^{2N} Y$ is not isomorphic to any subspace of $X^*$.  \end{enumerate}
\label{thm2}
\end{theorem}

\begin{proof}For $(i)$, we combine Lemma \ref{mltdwn} with Lemma \ref{2kemhi} applied with $\beta=3$ and $\gamma=2$ to find a tree $S$ with $\text{rank}(S)=\omega^N$ and collections $(y_{j,s})_{s\in S}\subset 3B_Y$, $(\psi_{j,s})_{s\in S}\subset Y^*$, $1\leqslant j\leqslant 2N+1$, such that $(\otimes_{j=1}^{2N+1} y_{j,s})_{s\in S}$ is weakly null, $\langle \psi_{j,s}, y_{j,s}\rangle=1$ for all $1\leqslant j\leqslant 2N+1$ and $s\in S$, $\langle \otimes_{j=1}^{2N+1} \psi_{\sigma(j),s'}, \otimes_{j=1}^{2N+1} y_{\tau(j), s}\rangle=0$ if $\sigma, \tau$ are distinct permutations on $\{1, \ldots, 2N+1\}$ and $s\not\sim s'$, and $\|\sum_{r\preceq s}\otimes_{j=1}^{2N+1} y_{j,r}\|_{\ee,1}^w \leqslant 2$ for all $s\in S$.    For $s\in MAX(S)$ and $r\preceq s$, 

\begin{align*} \Bigl\langle\sum_{s'\preceq s} S_{2N+1} \otimes_{j=1}^{2N+1} \psi_{j,s'}, S_{2N+1}\otimes_{j=1}^{2N+1} y_{j,r}\Bigr\rangle & =\sum_{s'\preceq s} \frac{1}{(2N+1)!^2}\sum_{\sigma, \tau} \langle \otimes_{j=1}^{2N+1} \psi_{\sigma(j), s'}, \otimes_{j=1}^{2N+1} y_{\tau(j),r}\rangle \\ & = \frac{1}{(2N+1)!^2}\sum_{\sigma, \tau} \langle \otimes_{j=1}^{2N+1} \psi_{\sigma(j), s'}, \otimes_{j=1}^{2N+1} y_{\tau(j),r}\rangle \\ & = \frac{1}{(2N+1)!^2} \sum_\sigma \langle \otimes_{j=1}^{2N+1} \psi_{\sigma(j), s'}, \otimes_{j=1}^{2N+1} y_{\sigma(j), r}\rangle \\ & = \frac{1}{(2N+1)!}\langle \otimes_{j=1}^{2N+1} \psi_{j,r}, \otimes_{j=1}^{2N+1} y_{j,r}\rangle \\ & =\frac{1}{(2N+1)!}.\end{align*}  Since \[\|S_{2N+1}\sum_{s'\preceq s} \otimes_{j=1}^{2N+1} \psi_{j,s'}\|_{\ee,s} \leqslant \|\sum_{s'\preceq s} \otimes_{j=1}^{2N+1} \psi_{j,s'}\|_\ee \leqslant \|(\otimes_{j=1}^{2N+1} \psi_{j,s'})_{s'\preceq s}\|_{\ee,1}^w \leqslant 2,\]  $\|S_{2N+1}\otimes_{j=1}^{2N+1} y_{j,r}\|_{\pi,s} \leqslant \|S_{2N+1}\|3^{2N+1}$, and $(\otimes_{j=1}^{2N+1} y_{j,r})_{r\in S}$ is weakly null, it follows that \[(\frac{1}{\|S_{2N+1}\|3^{2N+1}} S_{2N+1}\otimes_{j=1}^{2N+1} y_{j,r})_{r\in S}\subset B_{\tim_{\pi,s}^{2N+1} Y}\] is weakly null and for any $s\in S$ and $y\in \text{co}\{\frac{1}{\|S_{2N+1}\|3^{2N+1}} S_{2N+1} \otimes_{j=1}^{2N+1} y_{j,r}: r\preceq s\}$, \[\|y\|_{\pi,s} \geqslant \|y\|_\pi \geqslant \frac{1}{2}\Bigl\langle \sum_{s'\preceq s} S_{2N+1}\otimes_{j=1}^{2N+1} \psi_{j,s'}, y\Bigr\rangle \geqslant \frac{1}{2 \|S_{2N+1}\|3^{2N+1}(2N+1)!}.\] We conclude that $Sz(\tim_{\pi,s}^{2N+1} Y)>\omega^N$ as in Theorem \ref{thm}, which finishes $(i)$. 

For $(iii)$, we argue similarly, except each $2N+1$ above should be replaced by $2N$, the tree $S$ can be taken to have rank $\omega^{N-1}n$ rather than $\omega^N$, and we have the upper estimate $\|S_{2N}\sum_{s'\preceq s}\otimes_{j=1}^{2N} \psi_{j,s'}\|_{\ee,s} \leqslant 2 \cdot n^{1/2}$.   Thus for each $n\in\nn$, we obtain a weakly null collection $(\frac{1}{\|S_{2N}\|3^{2N}} S_{2N}\otimes_{j=1}^{2N}  y_{j,r})_{r\in S} \subset B_{\tim_{\pi,s}^{2N}Y}$ such that for any $s\in S$ and $y$ which is a convex combination of a branch of the collection, $\|y\|_{\pi,s}\geqslant \frac{1}{2 n^{1/2} \|S_{2N}\|3^{2N}}$.  As in Theorem \ref{thm}, it follows that $\tim_{\pi,s}^{2N}Y$ cannot have property $\mathfrak{A}_{N-1,q}$ for any $2<q\leqslant \infty$. This finishes $(iii)$. 

For $(ii)$, we can find a tree $S$ with $\text{rank}(S)=\omega^N$ and $(y_{j,s})_{s\in S}\subset Y$, $(\psi_{j,s})_{s\in S}\subset Y^*$, $1\leqslant j\leqslant 2N+1$,  as in the proof of $(i)$, except we reverse the inequalities: $\|\psi_{j,s}\|\leqslant 3$ for all $s\in S$ and $1\leqslant j\leqslant 2N+1$, and $\|(\otimes_{j=1}^{2N+1} y_{j,s'})_{s' \preceq s}\|_{\ee,1}^w$. The latter inequality implies that for any $s\in S$, $\|S_{2N+1}\sum_{s'\preceq s} \otimes_{j=1}^{2N+1} y_{j,s'}\|_{\ee,s} \leqslant \|\sum_{s'\preceq s} \otimes_{j=1}^{2N+1} y_{j,s}\|_\ee\leqslant 2$.  If $X$ is a Banach space and $A:\tim_{\ee,s}^{2N+1} Y\to X^*$ is an isomorphic embedding, we can assume without loss of generality that $c\|u\|\leqslant \|Au\|\leqslant \|u\|$ for all $u\in \tim_{\ee,s}^{2N+1}Y$.     Define $\xi_\varnothing=0$ and for $s\in S$, $\xi_s=\frac{1}{2}\sum_{s\preceq s'}AS_{2N+1}\otimes_{j=1}^{2N+1} y_{j,s}$.   Note that for any $s\in S$, \begin{align*} 2c^{-1}\|\xi_s-\xi_{s^-}\| & \geqslant \Bigl\|S_{2N+1}\sum_{s'\preceq s}\otimes_{j=1}^{2N+1} y_{j,s'}-S_{2N+1}\sum_{s'\preceq s^-} \otimes_{j=1}^{2N+1} y_{j,s'}\Bigr\| = \|S_{2N+1}\otimes_{j=1}^{2N+1} y_{j,s}\| \\ & \geqslant \Bigl\langle \frac{1}{\|S_{2N+1}\|3^{2N+1}} S_{2N+1} \otimes_{j=1}^{2N+1} \psi_{j,s}, S_{2N+1}\otimes_{j=1}^{2N+1} y_{j,s}\Bigr\rangle = \frac{1}{\|S_{2N+1}\|3^{2N+1}(2N+1)!}.\end{align*}  Arguing as in Theorem \ref{thm}, for all $0\leqslant \zeta\leqslant \omega^N$ and $s\in (\{\varnothing\}\cup S)^\zeta$, $\xi_s\in s_\ee^\zeta(B_{X^*})$ for any $0<\ee<\frac{c}{2\|S_{2N+1}\|3^{2N+1}(2N+1)!}$.  From this it  follows  that for such $\ee$, $0\in s^{\omega^N}_\ee(B_{X^*})$ and  $Sz(X)>\omega^N$.

For $(iv)$, we argue similarly. The tree $S$ can be taken to have rank $\omega^{N-1}n$ for any $n$, and the definition of $\xi_s$ should include a factor of $n^{-1/2}$ not present in the previous paragraph. Assuming $A:\tim_{\ee,s}^{2N}Y\to X^*$ is an embedding as in the previous paragraph,  we deduce that for any $n\in\nn$ and $0<\ee < \frac{c}{2n^{1/2}\|S_{2N}\|3^{2N}(2N)!}$, $s^{\omega^{N-1}n}_\ee (B_{X^*})\neq\varnothing$. We argue as in Theorem \ref{thm} to deduce that $X$ therefore cannot have property $\mathfrak{A}_{N-1,q}$ for any $2<q\leqslant \infty$.

\end{proof}

\section{Main results} \label{sec: mainresults}

\begin{theorem}Fix $N\in\nn$ and $M>N$.   Let $X_1, \ldots, X_N$ be Banach spaces with TAP and $\mathfrak{A}_{0,\infty}$  such that $X^*_i$ has non-trivial cotype for each $1\leqslant i\leqslant N$. Let $Y,Y_1, \ldots, Y_M$ be any infinite dimensional Banach spaces.  Then neither $\tim_{\pi,i=1}^M Y_i$ nor $\tim_{\pi,s}^MY$ is isomorphic to any subspace of any quotient of $\tim_{\pi,i=1}^N X_i$,  and neither $\tim_{\ee,i=1}^M Y_i$ nor $\tim_{\ee,s}^M Y$ is  isomorphic to any subspace of $\tim_{\ee,i=1}^N X^*_i$. 

\label{sharp1}
\end{theorem}

\begin{proof} We treat the even and odd cases separately. In the odd case, we replace $N$ with $2N-1$. Suppose that $M>2N-1$. Let $X=\tim_{\pi,i=1}^{2N-1} X_i$ and note that $X^*=\tim_{\ee,i=1}^{2N-1} X^*_i$ by Corollary \ref{sparse}. By Theorem \ref{upper1}, $X$ has property $\mathfrak{A}_{N-1,\infty}$. By Theorem \ref{thm}, $\tim_{\pi,i=1}^M Y_i$ is not isomorphic to any subspace of any quotient of $X$ and $\tim_{\ee,i=1}^M Y_i$ is not isomorphic to any subspace of $X^*=\tim_\ee^N X_i^*$.  The statement for the symmetric, projective tensor product is similar. 

Now we treat the even case and replace $N$ in the statement of the corollary with $2N$. Let $X=\tim_{\pi,i=1}^{2N} X_i$ and note that $X^*=\tim_{\ee,i=1}^{2N} X^*_i$ by Corollary \ref{sparse}. By Theorem \ref{upper1}, $Sz(X)\leqslant \omega^N$. By Theorem \ref{thm}, $\tim_{\pi,i=1}^M Y_i$ is not isomorphic to any subspace of any quotient of $X$ and $\tim_{\ee,i=1}^M Y_i$ is not isomorphic to any subspace of $X^*=\tim_\ee^N X_i^*$.  The statement for the symmetric, injective tensor product is similar. 

\end{proof}

\begin{rem}\upshape Recall that Tsirelson's space $T$ 
\cite{T} is reflexive with a basis and $T^*$ has cotype $q$ for every $2<q<\infty$ \cite[Page 117]{CS}. The Argyros-Gasparis-Motakis space $\mathfrak{X}$ is asymptotic $c_0$ with shrinking basis (and containing no isomorphic copy of $c_0$) such that $\mathfrak{X}^*\approx \ell_1$.   If $K$ is compact, Hausdorff with finite Cantor-Bendixson index, then $C(K)$ has TAP and  $\mathfrak{T}_{0,\infty}$ and $C(K)^*\approx \ell_1(K)$.    If $\Gamma$ is any set, then $c_0(\Gamma)$ has $\mathfrak{T}_{0,\infty}$ and $c_0(\Gamma)=\ell_1(\Gamma)$.  If $X$ is any of these spaces, then for $M>N$, $\tim_\pi^M X$ is not isomorphic to any subspace of any quotient of $\tim_\pi^N X$ and $\tim_\ee^M X^*$ is not isomorphic to any subspace of $\tim_\ee^N X^*$. Our next result isolates a particular case of these results and the symmetric analogues of these statements.

\end{rem}

The next corollary answers a question raised in \cite{CGS}. 

\begin{corollary}Fix natural numbers $M,N$ with $N>M$.  \begin{enumerate}[label=(\roman*)]\item $\tim_{\pi,s}^N c_0$ is not isomorphic to any subspace of any quotient of $\tim_\pi^M c_0$. \item $\tim_\pi^N c_0$ is not isomorphic to any subspace of any quotient of $\tim_\pi^M c_0$. \item $\tim_{\ee,s}^N \ell_1$ is not isomorphic to any subspace of $\tim_\ee^M \ell_1$. \item $\tim_\ee^N\ell_1$ is not isomorphic to any subspace of $\tim_\ee^M \ell_1$. \end{enumerate}
\label{sharp2}
\end{corollary}

\begin{corollary} If $Y, Y_1, \ldots, Y_{2N}$ are infinite dimensional Banach spaces such that $Y, Y_1, \ldots, Y_{2N}$ have TAP and property $\mathfrak{A}_{0,\infty}$, and such that $Y^*, Y^*_1, \ldots, Y^*_{2N}$ each have non-trivial cotype.     Then \[Sz(\tim_{\pi,s}^{2N-1}Y)=Sz(\tim_{\pi,s}^{2N}Y) = Sz(\tim_{\pi,i=1}^{2N-1} Y_i)=Sz(\tim_{\pi,i=1}^{2N}Y_)=\omega^N.\] 

\label{sharp3}
\end{corollary}

We note that the Szlenk index is too coarse a tool to prove that $\tim_{\pi,i=1}^{2N}Y_i$ is not isomorphic to a subspace of a quotient of $\tim_{\pi,i=1}^{2N-1}Y_i$. The further stratifications $\mathfrak{A}_{N-1,p}$, $1<p\leqslant \infty$, were necessary.

\begin{corollary} If $K_1, \ldots, K_N, L_1, \ldots, L_M,L$ are scattered, compact, Hausdorff spaces and $M>N$, then neither $\otimes_{\pi,i=1}^M C(L_i)$ nor $\otimes_{\pi,s}^M C(L)$ is isomorphic to any quotient of $\otimes_{\pi,i=1}^N C(K_i)$. 

\end{corollary}

\begin{proof} If $\otimes_{\pi,i=1}^M C(L_i)$ were isomorphic to a quotient of $\otimes_{\pi,i=1}^N C(K_i)$, then $\otimes_{\ee,i=1}^M \ell_1(L_i)$ would be isomorphic to a subspace of $\otimes_{\ee,i=1}^N \ell_1(K_i)$.    Here we are using Corollary \ref{easy2} and Corollary \ref{sharp1}.  The  statement for $\otimes_{\pi,s}^M C(L)$ is similar.

\end{proof}

\end{document}